\documentclass[a4paper, 12pt, oneside]{amsart}
\usepackage[utf8]{inputenc}
\usepackage{amsmath,amsthm,amsfonts,amssymb}
\usepackage{url}
\usepackage{changepage}

\usepackage{graphicx} 

\usepackage{mathtools} 

\usepackage[shortlabels]{enumitem}

\usepackage{tabularx}
\usepackage{tabularray}
\usepackage{geometry}
\usepackage{tikz-cd}

\usepackage{stmaryrd}

\usepackage{mathrsfs}

\usepackage[hidelinks]{hyperref}

\DeclareMathOperator{\ham}{Ham}

\DeclareMathOperator{\ske}{skew}

\DeclareMathOperator{\en}{\mathrm{End}}

\DeclarePairedDelimiter\pg{[}{]}

\newcommand{\rest}[2]{\left.#1\right|_{#2}}

\newcommand{\dif}{\mathrm{d}}

\newcommand{\spann}{\mathrm{span}}

\newcommand{\pr}{\mathrm{pr}}
\newcommand{\kil}{\mathrm{Kill}}
\newcommand{\Aff}{\mathrm{Aff}}
\newcommand{\aff}{\mathrm{aff}}

\newcommand{\der}{\mathrm{Der}}
\newcommand{\gder}{\mathrm{gDer}}

\newcommand{\lcn}[1]{\prescript{#1}{}{\nabla}}

\newcommand{\cyc}{\mathrm{cyclic}}

\newcommand\varlist {,\makebox[1em][c]{.\hfil.\hfil.},}

\newcommand{\nb}{_{\scalebox{0.45}{$\nabla$}}}
\newcommand{\s}[1]{\scalebox{0.5}{#1}}

\DeclareMathOperator{\Img}{Im}

\newcommand{\N}{{\mathbb{N}}}

\newcommand{\R}{{\mathbb{R}}}

\newcommand{\cC}{\mathcal{C}}

\newcommand{\cCi}{\cC^\infty}

\newcommand{\la}{\langle}
\newcommand{\ra}{\rangle}

\newcommand{\GL}{\mathrm{GL}}

\DeclareMathOperator{\grad}{grad}

\DeclareMathOperator{\id}{id}

\DeclareMathOperator{\Der}{Der}

\DeclareMathOperator{\Sym}{Sym}
\DeclareMathOperator{\sym}{sym}

\newtheorem{theorem}{Theorem}[section]
\newtheorem{corollary}[theorem]{Corollary}
\newtheorem{lemma}[theorem]{Lemma}
\newtheorem{proposition}[theorem]{Proposition}
\newtheorem{example}[theorem]{Example}

\newtheorem{innercustomthm}{Theorem}
\newenvironment{customthm}[1]
{\renewcommand\theinnercustomthm{#1}\innercustomthm}
{\endinnercustomthm}

\newtheorem{innercustomprop}{Proposition}
\newenvironment{customprop}[1]
{\renewcommand\theinnercustomprop{#1}\innercustomprop}
{\endinnercustomprop}

\theoremstyle{definition}
\newtheorem{definition}[theorem]{Definition}
\theoremstyle{remark}

\newtheorem{remark}[theorem]{Remark}

\renewcommand{\pounds}{L}

\title[]{Symmetric Cartan calculus,\\ the Patterson-Walker metric\\ and Killing vector fields}

\author[F. Moučka]{Filip Moučka}
\author[R. Rubio]{Roberto Rubio}

\address{Universitat Aut\`onoma de Barcelona, 08193 Barcelona, Spain;\newline \indent Faculty of Nuclear Sciences and Physical Engineering, Czech Technical\newline\indent University in Prague, 115 19 Prague 1, Czech Republic}
\email{filip.moucka@autonoma.cat; mouckfil@cvut.cz}

\address{Universitat Aut\`onoma de Barcelona, 08193 Barcelona, Spain}
\email{roberto.rubio@uab.es}

\thanks{This project has been supported by MICIU/AEI/10.13039/501100011033 and the EU FEDER under the grants PID2022-137667NA-I00 (GENTLE) and CNS2024-154695 (DÉCOLLAGE). The first author has been partially supported by the Grant Agency of the Czech Technical University in Prague, grant No. SGS25/163/OHK4/3T/14. The second author has also received support from the 
MICIU/AEI and the EU FSE under the Ramón y Cajal fellowship RYC2020-030114-I and from the AGAUR under the grant 2021-SGR-01015. }

\begin{document}

\begin{abstract} We develop symmetric Cartan calculus, an analogue of classical Cartan calculus for symmetric differential forms. We first show that the analogue of the exterior derivative, the symmetric derivative, is not unique and its different choices are parametrized by torsion-free affine connections. We use a choice of symmetric derivative to generate the symmetric Lie derivative and the symmetric bracket, and give geometric interpretations of all of them. By proving the structural identities and describing the role of affine morphisms, we reveal an unexpected link of symmetric Cartan calculus with the Patterson-Walker metric, which we recast as a direct analogue of the canonical symplectic form on the cotangent bundle. We show that, in the light of the Patterson-Walker metric, symmetric Cartan calculus becomes a complete analogue of classical Cartan calculus. In this analogy, its Killing vector fields play a central role. 
\end{abstract}

\maketitle


\setcounter{tocdepth}{1}
\tableofcontents

\section{Introduction}

Classical Cartan calculus, understood as the calculus of vector fields $\mathfrak{X}(M)$ and differential forms $\Omega^\bullet(M)$ on a manifold $M$, is an ubiquitous language for differential geometry. Its structural identities are the graded commutators of the contraction operator $\iota$, the exterior derivative $d$ and the Lie derivative $L$. Namely, for $X,Y\in\mathfrak{X}(M)$, we have:
\begin{equation}\label{eq:Cartan-calculus}
 \begin{aligned}
 [\iota_X,\iota_Y]_\text{g}&=0, & & & & &[\dif,\dif]_\text{g}&=0, & & & & &[\iota_X,\dif]_\text{g}&=L_X,\\
 [L_X,\iota_Y]_\text{g}&=\iota_{[X,Y]}, & & & & & [L_X,L_Y]_\text{g}&=L_{[X,Y]}, & & & & & [\dif, L_X]_\text{g}&=0.
\end{aligned}
\end{equation}


The exterior derivative plays a geometric generating role. Together with the contraction, which is a linear-algebraic operation (that is, at the level of bundles, not necessarily sections), it recovers the Lie derivative and also the Lie bracket as a derived bracket. 

A very natural question is whether there is an analogous framework for symmetric differential forms and what it might be helpful for. This would first mean finding analogues to the exterior derivative, the Lie derivative and the Lie bracket satisfying similar properties. Surprisingly, the only precedent we are aware of is the definition of a symmetric bracket in \cite{CroGSST} (later geometrically interpreted in \cite{LewACD}), which was used, by mimicking the explicit formulas in the usual setting, to define a symmetric differential and a symmetric Lie derivative \cite{HeydSD}. We have found no justification that these are the only possible or even reasonable choices, no exhaustive study of the structural identities analogous to \eqref{eq:Cartan-calculus} and, beyond the symmetric bracket \cite{LewGISP, Lewbook}, neither a full geometric interpretation nor a significant application.

The aim of the present work is to develop and justify a sound and geometrically meaningful approach to symmetric Cartan calculus, as well as to show its two-way interaction with the Patterson-Walker metric. Key to our strategy is our starting point: we look for all possible symmetric analogues of the exterior derivative that can play an equivalent generating role as a geometric derivation (Definition \ref{def:geo-der}).

The first main feature of symmetric Cartan calculus is, precisely, that there is no canonical symmetric derivative for symmetric differential forms, which we denote by $\Upsilon^\bullet(M)$. Such an analogue depends on the choice of a connection, which can be, without loss of generality, assumed to be torsion free. In fact, we prove: 

\begin{customprop}{\ref{prop:corr-torsion-free-geo-der}}
 The map $\nabla\mapsto \nabla^s$ gives an isomorphism of $\Upsilon^2(M,TM)$-affine spaces:
 \begin{equation*}
\left\{\begin{array}{c}
 \text{torsion-free}\\
 \text{connections on } M
 \end{array}\right\}\overset{\sim}{\longleftrightarrow}\left\{\begin{array}{c}
 \text{geometric derivations of the algebra}\\
 \text{of symmetric forms on }M 
 \end{array}\right\}.
\end{equation*}
\end{customprop}

An a priori surprising fact is that $\nabla^s$ does not square to zero. However, this is natural, since the symmetric derivative is a usual derivation and derivations are an algebra for the usual commutator, as opposed to the graded commutator for graded derivations. Thus, squaring to zero is natural in the context of classical Cartan calculus since $[\dif,\dif]_g=2\,\dif\circ \dif$, whereas in the symmetric case, we have that the analogous expression of squaring to zero is $[\nabla^s,\nabla^s]=0$ for the usual commutator, which is trivially satisfied.

We give a geometric interpretation of $\nabla^s$ making use of the correspondence between a symmetric form $\varphi$ and a polynomial in velocities $\tilde{\varphi}$ (an element of a subspace of $\cCi(TM)$, as we explain in Section \ref{sec:geo-int-sym-der}), together with the geodesic spray $X_\nabla\in\mathfrak{X}(TM)$.

\begin{customprop}{\ref{prop: nablas polynomials}}
 Let $\nabla$ be a connection on $M$. For every $\varphi\in\Upsilon^r(M)$ we have
 \begin{equation*}
 \widetilde{\nabla^s\varphi}=X_\nabla\tilde{\varphi}.
 \end{equation*}
\end{customprop}

Closed symmetric forms for the symmetric derivative $\nabla^s$ recover the notion of Killing tensors. In the case of the Levi-Civita connection of a (pseudo-)Riemannian metric, Killing tensors are well established and naturally generalize Killing vector fields. They are related to the the symmetries of the Laplacian \cite{EasHSL}, or separability of Hamilton-Jacobi equation \cite{BenSRM}. Probably the most famous non-trivial Killing tensor is the \textit{Carter tensor} in the \textit{Kerr-Newman spacetime}, the spacetime of a charged rotating black hole \cite{CarA,CarB, WalOQ}. 

The symmetric Lie derivative with respect to a vector field $X\in\mathfrak{X}(M)$ is introduced by $\pounds_X^s=[\iota_X,\nabla^s]$ and we show in Proposition \ref{symLflow} (see also Figure \ref{UpsT}) that it captures the variation with respect to a combination of the flow of $X$ and the parallel transport with respect to $\nabla$, which we call the parallel flow (Definition \ref{def: par-flow}). Finally, the symmetric bracket is derived from the formula $[L^s_X,\iota_Y]=\iota_{[ X,Y]_s}$. Its geometric interpretation in relation to geodesically invariant distributions is recalled in Section \ref{sec:geo-int-sym-bracket}. Overall, we have the following properties for the (ungraded) commutator:
\begin{align*}
 [\iota_X,\iota_Y]&=0,& [\nabla^s,\nabla^s]&=0,&
 [\iota_X,\nabla^s]&=L^s_X,\\
 [L^s_X,\iota_Y]&=\iota_{[ X,Y]_s}, & [L^s_X,L^s_Y]&\neq L^s_{[X,Y]_s} & [\nabla^s, L^s_X]&\neq 0.
\end{align*}
Note that the analogues of the last two formulae of \eqref{eq:Cartan-calculus} are not true. Instead, we show in Section \ref{sec:commutation-relations} that
\begin{align*}
[L^s_X, L^s_Y]&=L^s_{[X,Y]}+\iota^s_{2(\nabla_{\nabla X}Y-\nabla_{\nabla Y}X-R_\nabla (X,Y))}, \\
[\nabla^s,L^s_X]&=2\nabla^s_{\nabla X}+\iota^s_{2\sym \iota_XR_\nabla +R^s_\nabla X}.
\end{align*}

To better understand the complicated terms on the right-hand sides, we recall in Section \ref{sec:affine-manifold-morphism} the Lie group of affine diffeomorphisms of a manifold with connection $\Aff(M,\nabla)$, the Lie algebra of affine vector fields $\aff(M,\nabla)$. We show that the abelian Lie subalgebra of parallel vector fields, $\aff_0(M,\nabla)\leq \aff(M,\nabla)$, can be characterized as those affine vector fields whose usual flow equals the parallel flow (Proposition \ref{prop: parallel-flow}). This leads us to identify the necessary and sufficient condition under which commutation relations of symmetric Cartan calculus simplify to a form that is completely analogous to classical Cartan calculus as in \eqref{eq:Cartan-calculus}. 

\begin{customthm}{\ref{thm: commutation-relations-sym-Cartan}}
 Let $\nabla$ be a torsion-free connection on $M$ and $X,Y\in\mathfrak{X}(M)$. All of the six relations
 \begin{align*}
 [\iota_X,\iota_Y]&=0,& [\nabla^s,\nabla^s]&=0,&
 [\iota_X,\nabla^s]&=L^s_X,\\
 [L^s_X,\iota_Y]&=\iota_{[ X,Y]_s}, & [L^s_X,L^s_Y]&=L^s_{[X,Y]_s}, & [\nabla^s, L^s_X]&=0
\end{align*}
are satisfied if and only if $X\in\aff_0(M,\nabla)$.
\end{customthm}


A further analysis of symmetric Cartan calculus reveals an unexpected and deep link with the Patterson-Walker metric $g\nb$ \cite{PatRE}, a metric on $T^*M$ that is determined by the choice of a torsion-free connection $\nabla$ on $M$. Using a lift $\hat{\nabla}$ of the connection $\nabla$ to $T^*M$ (Section \ref{sec:reformulation-sym-Cartan-calculus}), the symmetric derivative allows us to recast the \mbox{Patterson-Walker} metric as an analogue of the canonical symplectic form on~$T^*M$.
\begin{customthm}{\ref{thm: PW}}
 For any torsion-free connection $\nabla$ on $M$, there holds
 \begin{equation*}
\hat{\nabla}^s\alpha_\emph{can}=g\nb,
 \end{equation*}
 where $\alpha_\emph{can}$ is the canonical $1$-form on $T^*M$.
\end{customthm}

We next study Killing vector fields for the Patterson-Walker metric. These can be fully characterized in terms of tensor fields on the base manifold $M$, which was first shown in \cite{conformal-PW}. We formulate this result in a coordinate-free way and provide an alternative proof. 

\begin{customthm}{\ref{thm: Kill-PW}}
 Let $\nabla$ be a torsion-free connection on $M$. The vector space of Killing vector fields for $g\nb$ admits the natural decomposition:
 \begin{equation*}
 \{\alpha^\emph{v}\,|\,\alpha\in\kil^1(M,\nabla)\}\oplus \{X^c\,|\,X\in\aff(M,\nabla)\}\oplus\left\{\begin{array}{c|c}
 \!\!\pi^h & \begin{array}{c}
 \pi\in\mathfrak{X}^2(M) \text{ such that } \\
 \nabla\pi=0 \text{ and }R^\pi_\nabla=0
 \end{array}
 \end{array}\!\!\!\!\!\right\},
 \end{equation*}
 where $R^\pi_\nabla\in\Upsilon^2(M,\Sym^2TM)$ is given by
 \begin{equation*}
 R^\pi_\nabla(Y,Z)\coloneqq\sym(R\nb(Y,\pi(\,))Z+R\nb(Z,\pi(\,))Y).
 \end{equation*}
\end{customthm}

We extend Theorem \ref{thm: Kill-PW} by characterizing the Lie algebra structure and showing that the Killing vector fields for $g\nb$ form a $3$-graded Lie algebra.

\begin{customthm}{\ref{thm: PW-Lie}}
 Let $\nabla$ be a torsion-free connection on $M$. The decomposition in Theorem \ref{thm: Kill-PW} represents a natural grading of the Lie algebra of Killing vector fields for $g\nb$ concentrated in degrees $-1,0,1$. In particular, for $\alpha,\beta\in\kil^1(M,\nabla)$, $X,Y\in\aff(M,\nabla)$ and $\pi,\rho\in\mathfrak{X}^2(M,\nabla)$, we have
 \begin{align*}
 [\alpha^\emph{v},\beta^\emph{v}]&=0, & [X^c,\alpha^\emph{v}]&=(L_X\alpha)^\emph{v}, & [\alpha^\emph{v},\pi^h]&=\pi(\alpha)^c,\\
 [X^c,Y^c]&=[X,Y]^c, & [X^c,\pi^h]&=(L_X\pi)^h, & [\pi^h,\rho^h]&=0.
 \end{align*}
\end{customthm}


Finally, we use the description of Killing vector fields for the \mbox{Patterson-Walker} metric to prove our last result, which gives a conceptual explanation of the difference between symmetric Cartan calculus and classical Cartan calculus.
\begin{customthm}{\ref{thm: commutation-relations-sym-Cartan-PW}}
 Let $\nabla$ be a torsion-free connection on $M$ and $X,Y\in\mathfrak{X}(M)$. All of the six relations
 \begin{align*}
 [\iota_X,\iota_Y]&=0,& [\nabla^s,\nabla^s]&=0,&
 [\iota_X,\nabla^s]&=L^s_X,\\
 [L^s_X,\iota_Y]&=\iota_{[ X,Y]_s}, & [L^s_X,L^s_Y]&=L^s_{[X,Y]_s}, & [\nabla^s, L^s_X]&=0
\end{align*}
are satisfied if and only if $X^\text{c}$ is a gradient Killing vector field for $g\nb$.
\end{customthm}
This seemingly additional condition on $X$ implicitly appears in the classical case as well. Indeed, the analogues of gradient Killing vector fields for $g\nb$ are Hamiltonian ones for the canonical symplectic form $\omega_\text{can}$, and complete lifts are always Hamiltonian. This draws a complete analogy between Theorem \ref{thm: commutation-relations-sym-Cartan-PW} and~\eqref{eq:Cartan-calculus}.

The study of the lifts required to prove results of Section \ref{sec:Patterson-Walker} is recalled in Appendix \ref{app: lifts}. The development of this theory becomes quite technical, so we have made a special effort to provide clean and coordinate-free arguments.

We finish the paper with a table (Appendix \ref{sec:table}) that summarizes and compares various aspects of symmetric Cartan calculus with classical Cartan calculus.



Further applications of this new framework can be seen in other and future works. To start with, symmetric Cartan calculus plays a fundamental role in the development of symmetric Poisson geometry \cite{SymPoisson} and $C_n$-generalized geometry (to appear). On the other hand, just as the exterior derivative gives rise to de Rham cohomology, the question of whether there is any cohomology theory associated to the symmetric derivative, despite $\nabla^s$ not squaring to zero, arises naturally. The study of this cohomology theory will also be the subject of a forthcoming work.



\vspace{.52cm}


\textbf{Acknowledgements.} We would like to thank the Weizmann Institute of Science for their hospitality in hosting us as visitors, Maciej Dunajski for sharing with us his knowledge about the Patterson-Walker metric, Andrew Lewis for helpful discussions about the symmetric bracket and Rudolf Šmolka for noticing a mistake in the first version of this text. We also thank Arman Taghavi-Chabert and Josef Šilhan for bringing their work \cite{conformal-PW} to our attention.

\vspace{.52cm}

\textbf{Notation and conventions.} We work on the smooth category for manifolds and maps. We denote an arbitrary manifold (of positive dimension) by $M$, its smooth functions by $\cCi(M)$, its tangent bundle by $TM$, its cotangent bundle by $T^*M$, its vector fields by $\mathfrak{X}(M)$ and its differential forms by $\Omega^\bullet(M)$. 

When we consider a connection on $M$, we mean an affine connection (that is, a connection on $TM$), which can determine other connections (for instance, on $T^*M$ by duality). We will denote these connections with the same symbol.

By a vector space $V$, we mean a real finite-dimensional vector space. We use the polarization argument often: for two vector spaces $V$ and $V'$ over the same field, two totally symmetric multilinear maps $\varphi, \psi\colon V\times \ldots \times V\rightarrow V'$ coincide if and only if, for every $X\in V$, we have $ \varphi(X\varlist X)=\psi(X\varlist X).$


We use the Einstein summation convention throughout the text. 



\section{Symmetric Cartan calculus}

We first look at the linear-algebraic level before moving to smooth manifolds.

\subsection{The symmetric algebra and its derivations}\label{sec:sym-alg-derivations}

Consider a vector space $V$, its dual space $V^*$ and the tensor algebra $\otimes^\bullet V^*$. The natural action of a permutation $\pi\in \mathcal{S}_r$ on a decomposable element (with $\alpha^j\in V^*$),
\[
\pi\cdot (\alpha^1\otimes \ldots \otimes \alpha^r) \coloneqq \alpha^{\pi 1}\otimes \ldots \otimes \alpha^{\pi r},
\]
is extended linearly to $\otimes^r V^*$ and gives rise to the \textbf{symmetric projection} $\sym_r\colon \otimes^r V^*\rightarrow\otimes^r V^*$, 
\begin{align*}
\sym_r(\varphi)\coloneqq \frac{1}{r!}\sum_{\pi\in\mathcal{S}_r}\pi\cdot \varphi,
\end{align*}
and, more generally, $\sym\coloneqq \oplus_{r\in \N} \sym_r$. Consider the $\N$-graded vector space
\[
\Sym^\bullet V^*\coloneqq \frac{\otimes^\bullet V^*}{\ker \sym}\cong \Img \sym.
\]
The latter isomorphism allows us to realize $\Sym^\bullet V^*$ as a vector subspace of $\otimes^\bullet V^*$. Define the \textbf{symmetric product} of $\varphi\in\Sym^r V^*$ and $\psi\in\Sym^l V^*$ by
\begin{equation*}
\varphi\odot \psi\coloneqq \frac{(r+l)!}{r!\,l!}\sym(\varphi\otimes \psi).
\end{equation*}
We thus have the \textbf{symmetric algebra} $(\Sym^\bullet V^*,+,\odot)$.

At this linear-algebraic level, we can already look at the \textbf{contraction operator}. The degree-$(-1)$ map $\iota_X$, for any $X\in V$, descends from $\otimes^\bullet V^*$ to $\Sym^\bullet V^*$, where it becomes a derivation, that is,
\begin{equation}\label{eq:iota-X-derivation}
 \iota_X(\varphi\odot \psi)=\iota_X\varphi \odot \psi + \varphi\odot \iota_X\psi,
\end{equation} 
as we show in the proof of the next lemma. Note that the sign is always $+$, that is, it is not graded.

\begin{lemma}\label{lem:-1-der-are-contractions}
 The degree-$(-1)$ derivations of $\Sym^\bullet V^*$ are in bijective correspondence with $V$ via the contraction operator $\iota$. 
\end{lemma}

\begin{proof}
 \sloppy We first prove that the contraction operator is indeed a derivation of $\Sym^\bullet V^*$. Given $\varphi\in\Sym^r V^*$ and $\psi\in\Sym^l V^*$, we denote $n\coloneqq r+l-1$. We shall use the shorthand notation $\varphi(Y)\coloneqq \varphi(Y\varlist Y)$ for a symmetric $r$-tensor $\varphi$ acting on $r$ copies of a vector $Y\in V$. Then, for all $Y\in V$,
 \begin{equation*}
 (\iota_X(\varphi\odot\psi))(Y)=\frac{1}{r!l!}\sum_{\pi\in\mathcal{S}_n}\iota_X(\pi\cdot(\varphi\otimes \psi))(Y).
\end{equation*}
We can write the symmetric group $\mathcal{S}_n$ as the disjoint union of the two subsets $\mathcal{S}^r_n\coloneqq \lbrace \pi\in\mathcal{S}_n\,|\, 1\in\pi(\lbrace 1\varlist r\rbrace)\rbrace$ and $\mathcal{S}^l_n\coloneqq \lbrace \pi\in\mathcal{S}_n\,|\, 1\in\pi(\lbrace r+1\varlist r+l\rbrace)\rbrace$. Therefore,
\begin{equation*}
 (\iota_X(\varphi\odot\psi))(Y)=\frac{1}{r!l!}\Big(\sum_{\pi\in\mathcal{S}^r_n}(\iota_X\varphi)(Y)\,\psi(Y)+\sum_{\pi\in\mathcal{S}^l_n}\varphi(Y)\,(\iota_X\psi)(Y)\Big).
\end{equation*}
Since the cardinal of the sets $\mathcal{S}^r_n$ and $\mathcal{S}^l_n$ is $n!\,r$ and $n!\,l$, respectively, we arrive at
\begin{align*}
(\iota_X(\varphi\,\odot\,\psi))(Y)&=\frac{n!}{(r-1)!l!}(\iota_X\varphi)(Y)\,\psi(Y)+\frac{n!}{r!(l-1)!}\,\varphi(Y)\,(\iota_X\psi)(Y)\\
&=\frac{1}{(r-1)!l!}\sum_{\pi\in \mathcal{S}_n}(\iota_X\varphi)(Y)\,\psi(Y)+\frac{1}{r!(l-1)!}\sum_{\pi\in \mathcal{S}_n}\varphi(Y)\,(\iota_X\psi)(Y)\\
&=(\iota_X\varphi\odot\psi+\varphi\odot\iota_X\psi)(Y).
\end{align*}
Identity \eqref{eq:iota-X-derivation} then follows from polarization, so $\iota_X$ is a derivation. As an algebra, $\Sym^\bullet V^*$ is generated by $\Sym^0 V^*$, that is, the scalars (where degree-$(-1)$ maps act trivially), and $\Sym^1 V^*=V^*$, whose linear maps are exactly given by $(V^*)^*\cong V$ via the contraction $\iota$.
\end{proof}

\subsection{Geometric derivations} 
Given a manifold $M$, we replace $V$ in Section \ref{sec:sym-alg-derivations} by $T_m M$ and consider the $r$th-symmetric power bundle 
\[ \Sym^r T^*M, \]
whose sections we denote by $\Upsilon^r(M)$. We define the graded $\cCi(M)$-module
\[ 
\Upsilon^\bullet (M) \coloneqq \bigoplus_{r\in\N} \Upsilon^r(M).
\]

\begin{remark}
Note that $\Upsilon^0(M)=\cCi(M)$ and $\Upsilon^1 (M)=\Omega^1(M)$, we will use one notation or the other depending on the context. On the other hand, we have avoided talking about $\Sym^\bullet T^*M$ since it would be a bundle of infinite rank.
\end{remark}

\begin{remark}\label{rem:notation-upsilon-bundle}
It will be convenient to also introduce the notation
\[\Upsilon^r(M,E)\coloneqq \Gamma(\Sym^rT^*M\otimes E)\]
for an arbitrary vector bundle $E$ over $M$ and, analogously, $\Upsilon^\bullet(M,E)$. 
\end{remark}

The symmetric projection, the symmetric product and the contraction of Section \ref{sec:sym-alg-derivations} are easily extended to tensor fields. We thus obtain a graded algebra \[(\Upsilon^\bullet(M),+,\odot).\] The derivations of any graded algebra together with the commutator form a Lie algebra that is moreover graded. In this case, we denote it
by \( \Der(\Upsilon^\bullet(M)).\) For derivations of fixed degree, we will use a subindex, e.g., $\iota_X\in\Der_{-1}(\Upsilon^\bullet(M))$.

Since the exterior derivative, together with the contraction, generates the Cartan calculus for $\Omega^\bullet(M)$, now we look for analogous derivations of $\Upsilon^\bullet (M)$, which we call geometric.

\begin{definition}\label{def:geo-der}
A \textbf{geometric derivation} $D\in\der(\Upsilon^\bullet (M))$ is a derivation of degree $1$ satisfying, for all $f\in \cCi(M)$ and $X\in\mathfrak{X}(M)$,
\begin{equation*}
(Df)(X)=Xf.
\end{equation*}
\end{definition}

Every derivation of $\Upsilon^\bullet (M)$ is a local operator, and since the action on functions is prescribed, any geometric derivation is uniquely characterized by its action on $1$-forms. Namely, it is determined by a linear map $D\colon \Upsilon^1 (M)\rightarrow\Upsilon^2 (M)$ such that
\begin{equation}\label{Symderchar}
D(f\alpha)=\dif f\odot\alpha+f D\alpha
\end{equation} 
for all $f\in \cCi(M)$ and $\alpha\in\Upsilon^1 (M)=\Omega^1(M)$. Formula \eqref{Symderchar} resembles the defining property of a connection on the cotangent bundle. In fact, given such a connection $\nabla\colon \Omega^1(M)\rightarrow \Omega^1(M, T^*M)$, we define the linear map $\nabla^s\colon \Upsilon^1(M)\rightarrow\Upsilon^2(M)$ by
\begin{equation}\label{nablas}
\nabla^s\coloneqq 2\sym\circ\,\nabla,
\end{equation}
which satisfies 
\begin{equation*}
\nabla^s(f\alpha)=2\sym(\nabla(f\alpha))=2\sym(\dif f\otimes\alpha+f\nabla\alpha)=\dif f\odot \alpha+f\nabla^s\alpha.
\end{equation*}

\begin{definition}\label{symderdef}
Let $\nabla$ be a connection on $M$. The \textbf{symmetric derivative} corresponding to $\nabla$ is the geometric derivation $\nabla^s\in\der(\Upsilon^\bullet (M))$, defined by \eqref{nablas}.
\end{definition}

In fact, these are all the possible geometric derivations.

\begin{lemma}\label{lem:geo-der-are-sym-der}
Every geometric derivation is the symmetric derivative corresponding to some affine connection.
\end{lemma}

\begin{proof}
Consider an arbitrary geometric derivation $D$ and define 
\begin{equation*}
\nabla\coloneqq \frac{1}{2}\rest{(D+\dif)}{\Omega^1(M)},
\end{equation*}
which is a linear map $\Omega^1(M)\rightarrow\Omega^1(M,T^*M)$. Moreover, for all $f\in \cCi(M)$ and $\alpha\in\Omega^1(M)$,
\begin{align*}
\nabla(f\alpha)&=\frac{1}{2}(D(f\alpha)+\dif(f\alpha))=\frac{1}{2}(\dif f\odot\alpha+fD\alpha+\dif f\wedge\alpha+f\dif\alpha)\\
&=\frac{1}{2}(\dif f\otimes\alpha+\alpha\otimes\dif f+\dif f \otimes \alpha-\alpha\otimes\dif f)+f\nabla\alpha=\dif f\otimes\alpha+f\nabla\alpha,
\end{align*}
thus $\nabla$ is a connection. As $\dif\alpha\in\Omega^2(M)$, $\sym(\dif\alpha)=0$, so $\nabla^s=2\sym\circ\nabla=D$.
\end{proof}

To any connection $\nabla$ on $M$ with torsion $T_\nabla \in\Omega^2(M,TM)$ we can associate the torsion-free connection $\nabla^0$ given, for $X,Y\in\mathfrak{X}(M)$, by
\begin{equation*}
\nabla^0_XY\coloneqq \nabla_XY-\frac{1}{2}T_\nabla (X,Y).
\end{equation*}

\begin{lemma}\label{lem:sym-der-torsion-free}
 Given two affine connections, their symmetric derivatives are the same if and only if their associated torsion-free connections are the same.
\end{lemma}
\begin{proof}
Note that the symmetric derivative corresponding to a connection $\nabla$ on $M$ may be explicitly written as
\begin{equation*}
 (\nabla^s\alpha)(X,Y)=X\alpha(Y)+Y\alpha(X)-\alpha(\nabla_XY+\nabla_YX)
\end{equation*}
for all $\alpha\in\Upsilon^1(M)$ and $X,Y\in\mathfrak{X}(M)$. It follows that two connections $\nabla$ and $\nabla'$ on $M$ induce the same symmetric derivative if and only if
\begin{equation}\label{eq:sym-der-torsion-free}
 \nabla_XY+\nabla_YX=\nabla'_XY+\nabla'_YX.
\end{equation}
On the other hand, the associated torsion-free connection to $\nabla$ is given by
\begin{equation}\label{eq:assoc-tf-con}
 \nabla_X^0Y=\frac{1}{2}(\nabla_XY+\nabla_YX+[X,Y]),
\end{equation}
hence $\nabla^0={\nabla'}^0$ if and only if \eqref{eq:sym-der-torsion-free} is satisfied.
\end{proof}

\begin{remark}
 Note that, geometrically, having the same torsion-free connection means that they have the same geodesics. It may be seen easily from that both the geodesic equation and the associated torsion-free connection depend solely on the symmetric part of the original affine connection, see \eqref{eq:assoc-tf-con}. 
\end{remark}

\begin{proposition}\label{prop:corr-torsion-free-geo-der}
 The map $\nabla\mapsto \nabla^s$ gives an isomorphism of $\Upsilon^2(M,TM)$-affine spaces:
 \begin{equation*}
\left\{\begin{array}{c}
 \text{torsion-free}\\
 \text{connections on } M
 \end{array}\right\}\overset{\sim}{\longleftrightarrow}\left\{\begin{array}{c}
 \text{geometric derivations of the algebra}\\
 \text{of symmetric forms on }M 
 \end{array}\right\}.
\end{equation*}
\end{proposition}

\begin{proof}
 From Lemmas \ref{lem:geo-der-are-sym-der} and \ref{lem:sym-der-torsion-free}, there is a bijective correspondence. The set of torsion-free connections on $M$ is an affine space modelled on $\Upsilon^2(M,TM)$. On the other hand, for geometric derivations $D'$ and $D$,
\begin{equation*}
(D'-D)(f\alpha)=\dif f\odot\alpha+fD'\alpha-\dif f\odot\alpha-fD\alpha=f(D'-D)(\alpha)
\end{equation*}
for $f\in \cCi(M)$ and $\alpha\in\Upsilon^1(M)$, so the difference between $\rest{D'}{\Upsilon^1(M)}$ and $\rest{D}{\Upsilon^1(M)}$ may be viewed as an element of $\Upsilon^2(M,TM)$. Since every geometric derivation is uniquely determined by its restriction to $1$-forms, they are also a $\Upsilon^2(M,TM)$-affine space. The map $\nabla\mapsto \nabla^s$ is clearly affine and the result follows.
\end{proof}

Thus, without loss of generality, we can and will mostly consider (apart from some instances in Sections \ref{sec:affine-manifold-morphism}, \ref{sec:Patterson-Walker} and Appendix \ref{app: lifts}) connections to be torsion free.

\begin{remark}\label{rem: D^2}
 We have seen in the introduction that it is not natural for the replacement of the exterior derivative to square to zero. Actually, no geometric derivation squares to zero: for any $f\in \cCi(M)$, if $D\circ D=0$, we have
\begin{equation*}
0=D(Df^2)=D(2fDf)=2D f\odot Df+2fD(Df)=2\dif f\odot\dif f,
\end{equation*}
which is an contradiction (as $\dim M >0$), hence $D\circ D\neq 0$. In Appendix \ref{app: non-geometric}, we show that the same holds for non-geometric derivations.
\end{remark}

The unique extension of the symmetric derivative $\nabla^s$ defined by \eqref{nablas}
to $\Upsilon^\bullet(M)$ is given, for $\varphi\in\Upsilon^r(M)$, by
\begin{equation*}
 \nabla^s\varphi=(r+1)\sym(\nabla\varphi).
\end{equation*}
Indeed, the endomorphism on the right-hand side is of degree $1$, it coincides with the exterior derivative when $\varphi$ is a function and it reduces to \eqref{nablas} when $\varphi$ is a $1$-form. Moreover, it is a derivation (as the covariant derivative with respect to any vector field is), so it coincides with $\nabla^s$.

Explicitly for $\varphi\in\Upsilon^r(M)$ and $X_j\in\mathfrak{X}(M)$ we have
\begin{equation}\label{eq:nabla-s}
(\nabla^s\varphi)(X_1\varlist X_{r+1})=\sum_{j=1}^{r+1}(\nabla_{X_j}\varphi)(X_1\varlist \hat{X}_j\varlist X_{r+1}).
\end{equation}

\subsection{Geometric interpretation of the symmetric derivative}\label{sec:geo-int-sym-der}

We use the relation between symmetric forms on $M$ and a certain class of smooth functions on $TM$ to give a geometric interpretation of $\nabla^s$.


\begin{definition}\label{polvel}
 We call a smooth function $\xi\in \cCi(TM)$ a \textbf{degree}-$r$ \textbf{polynomial in velocities} on $M$ if it is a \textit{homogeneous function of degree $r$}, that is, for every $\lambda\in\R$ and $u\in TM$,
 \begin{equation*}
 \xi(\lambda\,u)=\lambda^r \xi(u).
 \end{equation*}
\end{definition}

 The direct sum of the spaces of degree-$r$ polynomials for $r\in \N$ is a subalgebra of $\cCi(TM)$ that is unital and graded. We call its elements \textbf{polynomials in velocities} on $M$. The next lemma explains this notation.

 \begin{lemma}\label{lem:polynomials}
 Let $\xi\in\cCi(TM)$ be a degree-$r$ polynomial in velocities on $M$. If $r=0$, there is a unique $f\in\cCi(M)$ such that $\xi=\pr^*f$. Whereas, for $r>0$, given a natural coordinate chart $(TU,\lbrace x^i\rbrace\cup\lbrace v^j\rbrace)$, there is a unique set of functions $\lbrace \xi_{j_1\dots j_r}\rbrace\subseteq \cCi(U)$ such that $\rest{\xi}{TU}=(\pr^*\xi_{j_1\dots j_r})\,v^{j_1}\dots v^{j_r}$. 
 \end{lemma}

 \begin{proof}
 Consider a degree-$0$ polynomial in velocities $\xi\in\cCi(TM)$. We have $\xi(\lambda u)=\xi(u)$, that is, $\xi$ is constant along the fibres of $TM$. Denote by $0\in\mathfrak{X}(M)$ the null vector field and define $f\in\cCi(M)$ by $f\coloneqq \xi\circ 0$. The uniqueness of $f$ follows from the surjectivity of $\pr$.
 
 Assume now that $r>0$. By the choice of a natural coordinate chart on $TM$, a degree-$r$ polynomial in velocities $\xi\in\cCi(TM)$ restricted to a fibre of $TM$ corresponds uniquely to a smooth homogeneous function $\xi_0\colon \R^n\rightarrow\R$ of degree $r$, where $n\coloneqq \dim M$. It is now enough to check that $\xi_0$ is a homogeneous polynomial of degree $r$. Acting by the partial derivative $\frac{\partial}{\partial u^i}$ on the relation $\xi_0(\lambda u^1\varlist\lambda u^n)=\lambda^r\xi_0(u^1\varlist u^n)$, we obtain $\lambda\frac{\partial \xi_0}{\partial u^i}(\lambda u^1\varlist \lambda u^n)=\lambda^r \frac{\partial \xi_0}{\partial u^i}( u^1\varlist u^n)$, hence $\frac{\partial \xi_0}{\partial u^i}\colon \R^n\rightarrow \R$ is a smooth homogeneous function of degree $r-1$. By \textit{Euler's homogenous function theorem}, every smooth homogenous function $\chi\colon \R^n\rightarrow \R$ of degree $r$ satisfies
 \begin{equation*}
 r\,\chi(u^1\varlist u^n)=\sum_{i=1}^nu^i\frac{\partial \chi}{\partial u^i}(u^1\varlist u^n).
 \end{equation*}
 Since we have already proven the statement for $r=0$, it follows, by induction on $r$, that $\xi_0$ is a homogeneous polynomial of degree $r$. Finally, the smoothness of $\xi$ gives the existence and smoothness of $\lbrace \xi_{j_1\dots j_r}\rbrace$ 
 \end{proof}

The assignment $\Upsilon^r(M)\rightarrow \cCi(TM)\colon \varphi\mapsto \tilde{\varphi}$, where, for $u\in TM$,
\begin{align*}
\tilde{\varphi}(u)&\coloneqq \frac{1}{r!}\,\varphi(u\varlist u) & &\text{for $r>0$},\\
\tilde{f}&\coloneqq \pr^*f & &\text{for $r=0$}
\end{align*}
gives, by Lemma \ref{lem:polynomials} and polarization, an isomorphism of unital graded algebras
\begin{equation}\label{eq: sym forms polynomials}
 \Upsilon^\bullet(M) \overset{\sim}{\longleftrightarrow}\left\{\begin{array}{c}
 \text{polynomials in velocities on } M
 \end{array}\right\}.
\end{equation}

We show that the symmetric derivative is closely related to the \textbf{geodesic spray} of $\nabla$, which is the vector field $X_\nabla\in\mathfrak{X}(TM)$ given, for $u\in TM$, by
\begin{equation*}
 (X_\nabla)_u=u^\text{h}\in T_uTM.
\end{equation*}
The superscript $\text{h}$ here stands for the \textit{horizontal lift} to $TM$
(for more details, see Appendix \ref{app: lifts}, where the analogous horizontal lift to $T^*M$ is recalled).

\begin{remark}
 In natural local coordinates $(TU,\lbrace x^i\rbrace\cup\lbrace v^j\rbrace)$, we have
 \begin{equation*}
 (X_\nabla)_u=v^i(u)\rest{\partial_{x^i}}{u}-v^i(u) v^j(u)\,\Gamma^k_{ij}(\pr(u))\rest{\partial_{v^k}}{u}.
 \end{equation*}
\end{remark}

\begin{proposition}\label{prop: nablas polynomials}
 Let $\nabla$ be a connection on $M$. For every $\varphi\in\Upsilon^r(M)$ we have
 \begin{equation*}
 \widetilde{\nabla^s\varphi}=X_\nabla\tilde{\varphi}.
 \end{equation*}
\end{proposition}

\begin{proof}
 Since $\nabla^s$ and $X_\nabla$ are derivations of $\Upsilon^\bullet (M)$ and $\cCi(TM)$ respectively. It is enough to check that the equality holds for functions and $1$-forms on $M$. For $f\in \cCi(M)$, we have
 \begin{equation*}
 (X_\nabla\tilde{f})(u)=(X_\nabla\pr^*f)(u)=uf=(\dif f)(u)=\widetilde{\nabla^s f}(u).
 \end{equation*}
 On the other hand, for $\alpha\in\Omega^1(M)$ locally given by $\rest{\alpha}{U}=\alpha_i\dif x^i$, we get
 \begin{align*}
 (X_\nabla\tilde{\alpha})(u)&=(X_\nabla)_u((\pr^*\alpha_k) v^k)=v^i(u)v^j(u)\Big(\frac{\partial \alpha_j}{\partial x^i}(\pr(u))-\alpha_k(\pr(u))\,\Gamma^k_{ij}(\pr(u))\Big)\\
 &=(\nabla_u\alpha)(u)=\frac{1}{2}(\nabla^s\alpha)(u,u)=\widetilde{\nabla^s\alpha}(u).
 \end{align*}
\end{proof}

\subsection{Killing tensors} 
Symmetric tensors that are $\nabla^s$-closed have been considered before under the name of Killing tensors. We give some examples and recall their geometric significance.

\begin{definition}\label{def: Killing tensors}
Given a connection $\nabla$ on $M$, a \textbf{Killing tensor} for $\nabla$ is a symmetric form $K\in\Upsilon^\bullet(M)$ such that 
\begin{equation*}
\nabla^sK=0.
\end{equation*}
\end{definition}

We denote the Killing tensors of homogeneous degree $r$ by $\kil^r(M,\nabla)$. They are a vector subspace, but not a $\cCi(M)$-submodule, of the space of symmetric forms $\Upsilon^r(M)$. Together with the symmetric product, they form a graded subalgebra
\[
\kil^\bullet(M,\nabla):=\bigoplus_{r\in\N} \kil^r(M,\nabla)\subseteq \Upsilon^\bullet(M),
\]
which we call the \textbf{Killing tensor algebra} of $\nabla$.

\begin{example}[Killing $0$-tensors]
 As the symmetric derivative is geometric (Definition \ref{def:geo-der}), the space of Killing $0$-tensors for any connection $\nabla$ is the space of locally constant functions.
\end{example}

\begin{example}\label{ex: Killing tensors LC}
 For the Levi-Civita connection of a (pseudo-)Riemannian metric $g$, by Proposition \ref{prop: nablasLC} below, $g^{-1}$ identifies $\kil^1(M,\lcn{g})$ with the Killing vector fields for $g$. Moreover, we always have that $g\in\kil^2(M,\lcn{g})$.
\end{example}

Let us recall that the geometric significance of Killing tensors is that they induce conserved quantities along geodesics. For this, we use a characterization of the geodesic spray $X_\nabla\in\mathfrak{X}(TM)$, which was introduced in Section \ref{sec:geo-int-sym-der}.

\begin{lemma}\label{lem: X nabla on geodesics}
 Let $\nabla$ be a connection on $M$. The geodesic spray is the unique vector field whose integral curve starting at $u\in TM$ is the curve $\dot{\gamma}:I\rightarrow TM$, where $\gamma:I\rightarrow M$ is the $\nabla$-geodesic satisfying $\dot{\gamma}(0)=u$. In particular, for any $\nabla$-geodesic $\gamma$ on $M$ and $\xi\in\cCi(TM)$ there holds
 \begin{equation*}
 (X_\nabla\xi)\circ\dot{\gamma}=\frac{\dif }{\dif t}(\xi\circ\dot{\gamma}).
 \end{equation*}
\end{lemma}

\begin{proof}
Consider a curve $V:I\rightarrow TM$. In natural coordinates $(TU, \lbrace x^i\rbrace\cup\lbrace v^j\rbrace)$,
 \begin{equation*}
 (X_\nabla)_{V(t)}=V^i(t)\rest{\partial_{x^i}}{V(t)}-V^i(t)V^j(t)\,\Gamma^k_{ij}(\gamma(t))\rest{\partial_{v^k}}{V(t)},
 \end{equation*}
where $V^i:=v^i\circ V$ and $\gamma:=\pr\circ V$. Using the notation $\gamma^i:=x^i\circ \gamma$, the equation for the integral curve of $X_\nabla$ starting at $u\in TM$ is locally given by
 \begin{align*}
 \dot{\gamma}^i(t)&=V^i(t), & \dot{V}^i(t)&=-V^i(t)V^j(t)\,\Gamma^k_{ij}(\gamma(t)), & V(0)&=u.
 \end{align*}
 Using the first equation on the second one, we recover the geodesic equation 
 \begin{align*}
 0&=\rest{\nabla_{\dot{\gamma}}\dot{\gamma}}{U}=\ddot{\gamma}^k(t)+\dot{\gamma}^i(t)\dot{\gamma}^j(t)\,\Gamma^k_{ij}(\gamma(t)), & \dot{\gamma}(0)&=u.
 \end{align*}
 The converse follows simply from the fact that every vector field is uniquely characterized by its integral curves.
\end{proof}

\begin{proposition}
\label{prop: Kill polynomials}
Let $\nabla$ be a connection on $M$. The assignment $K\mapsto \tilde{K}$ gives an isomorphism of unital graded algebras

\[ 
 \kil^\bullet(M,\nabla) \overset{\sim}{\longleftrightarrow}\left\{\begin{array}{c}
 \text{ polynomials in velocities} \text{ on } M
 \\ \text{ constant along }\nabla\text{-geodesics}
\end{array}\right\}.
\]
\end{proposition}

\begin{proof}
 By the correspondence \eqref{eq: sym forms polynomials} and Proposition \ref{prop: nablas polynomials}, we have $K\in\kil^\bullet(M,\nabla)$ if and only if $X_\nabla\tilde{K}=0$. In particular, we have $(X_\nabla\tilde{K})\circ\dot{\gamma}=0$ for every $\nabla$-geodesic $\gamma$ if $K\in\kil^\bullet(M,\nabla)$, which, by Lemma \ref{lem: X nabla on geodesics}, means that $\tilde{K}\circ\dot{\gamma}$ is constant. Conversely, assume that $\tilde{K}$ is constant along $\nabla$-geodesics, hence $(X_\nabla\tilde{K})\circ\dot{\gamma}=0$. Since there is a $\nabla$-geodesic $\gamma$ such that $\dot{\gamma}(0)=u$ for every $u\in TM$, we get $X_\nabla\tilde{K}=0$.
\end{proof}

\subsection{The symmetric Lie derivative} Once we have a characterization of the symmetric analogues of the exterior derivative, we can make the following definition, drawing inspiration from Cartan's magic formula and using the fact that the commutator of derivations is again a derivation.

\begin{definition}\label{symLalg}
The \textbf{symmetric Lie derivative} \mbox{$\pounds^s_X\in\der_0(\Upsilon^\bullet(M))$} corresponding to $\nabla$ with respect to $X\in\mathfrak{X}(M)$ is defined by
\begin{equation*}
L^s_X\coloneqq [\iota_X,\nabla^s]=\iota_X\circ\nabla^s-\nabla^s\circ\iota_X.
\end{equation*}
\end{definition}

\begin{remark}
 Note that $\pounds_X^sf=Xf$ for any $f\in \cCi(M)$ and, thanks to \eqref{eq:nabla-s}, an explicit formula is given, for $\varphi\in\Upsilon^r(M)$ and $X,X_j\in\mathfrak{X}(M)$, by
\begin{equation}\label{eq:formula-L-nabla-s}
(\pounds^s_X\varphi)(X_1\varlist X_r) =(\nabla_X\varphi)(X_1\varlist X_r) -\sum_{j=1}^r\varphi(\nabla_{X_j}X,X_1\varlist \hat{X}_j\varlist X_r). 
\end{equation}
\end{remark}

The symmetric Lie derivative is related to the covariant and the usual Lie derivatives as follows.

\begin{proposition}\label{covLL}
For any connection $\nabla$ on $M$ and $X\in\mathfrak{X}(M)$,
\begin{equation}\label{covLLeq}
\nabla^0_X=\frac{1}{2}(\pounds^s_X+\pounds_X)
\end{equation}
as elements in $\Der(\Upsilon^\bullet(M))$.
\end{proposition}

\begin{proof}
It is straightforward to check that the covariant and Lie derivatives define elements in $\Der(\Upsilon^\bullet(M))$. It is then enough to check the equality for a function $f\in\cCi(M)$ (which is trivial, as all the terms are $Xf$), and for a $1$-form $\alpha\in\Upsilon^1(M)$,
\begin{align*}
\frac{1}{2}(\pounds_X+\pounds^s_X)\,\alpha&=\frac{1}{2}(\iota_X\circ\dif+\dif\circ\iota_X+\iota_X\circ\nabla^s-\dif\circ\iota_X)\,\alpha\\
&=\iota_X(\ske\circ\nabla^0+\sym\circ\nabla^0)\,\alpha =\iota_X\nabla^0\alpha=\nabla^0_X\alpha,
\end{align*} 
where `$\ske$' denotes the usual skew-symmetric projection analogous to `$\sym$'.
\end{proof}

As a consequence, for a~torsion-free connection $\nabla$ and $\varphi\in\Upsilon^\bullet(M)$ such that $\nabla \varphi=0$, we have
\begin{equation}\label{eq: symLie and Lie}
\pounds_X^s\varphi=-\pounds_X\varphi.
\end{equation}
This applies, in particular, to a (pseudo-)Riemannian metric $\varphi=g$ and $\nabla=\lcn{g}$ (the Levi-Civita connection of $g$), which allows us to determine the symmetric derivative corresponding to $\lcn{g}$ in an elegant way.

\begin{proposition}\label{prop: nablasLC} 
 The symmetric derivative corresponding to the Levi-Civita connection
 $\lcn{g}$ of a (pseudo-)Riemannian metric $g$ is completely determined by
\begin{equation*}
\lcn{g}^s\alpha=\pounds_{g^{-1}\alpha}\,g
\end{equation*}
 for any $\alpha\in\Upsilon^1(M)$.
\end{proposition}

\begin{proof}
 Using the fact that $\lcn{g}g=0$, and hence $\lcn{g}^sg=0$, we find
 \begin{equation*}
 \lcn{g}^s\alpha=\lcn{g}^s(\iota_{g^{-1}\alpha}g)=-(\iota_{g^{-1}\alpha}\lcn{g}^s-\lcn{g}^s\iota_{g^{-1}\alpha})g=-L^s_{g^{-1}\alpha}g.
 \end{equation*}
 It follows from \eqref{eq: symLie and Lie} that $\lcn{g}^s\alpha=L_{g^{-1}\alpha}g$.
\end{proof}

\subsection{Geometric interpretation of the symmetric Lie derivative}\label{sec:geo-int-sym-Lie-der}

Given a vector field $X\in\mathfrak{X}(M)$, Lie and covariant derivatives at a point $m\in M$ have infinitesimal formulas in terms, respectively, of the local flow $\Psi^X$ and the parallel transport $P^\gamma$ along an integral curve $\gamma$ of $X$ starting at $m$. For a tensor field $\tau$, 
\begin{align}
(\pounds_X\tau)_m&=\lim_{t\to 0}\frac{1}{t}((\Psi^X_t)^*_m\tau_{\Psi^X_t(m)}-\tau_m)=\rest{\frac{\dif}{\dif t}}{t=0}(\Psi^X_t)^*_m\tau_{\Psi^X_t(m)}\label{Lieflow}\\
(\nabla_X\tau)_m&=\lim_{t\to 0}\frac{1}{t}(P^\gamma_{t,0}\tau_{\Psi^X_t(m)}-\tau_m)=\rest{\frac{\dif}{\dif t}}{t=0}P^\gamma_{t,0}\tau_{\Psi^X_t(m)}.\label{covflow}
\end{align}

The relation \eqref{covLLeq} gives us a recipe how to write a similar formula also for a symmetric Lie derivative.

\begin{proposition}\label{symLflow}
Let $\nabla$ be a connection on $M$. For all $\varphi\in\Upsilon^\bullet (M)$, $X\in\mathfrak{X}(M)$, and $m\in M$, there holds
\begin{equation*}
(\pounds^s_X\varphi)_m=\lim_{t\to 0}\frac{1}{t}(P^\gamma_{2t,0}(\Psi^X_{-t})_{\Psi^X_{2t}(m)}^*\varphi_{\Psi^X_t(m)}-\varphi_m)=\rest{\frac{\dif}{\dif t}}{t=0}P^\gamma_{2t,0}(\Psi^X_{-t})_{\Psi^X_{2t}(m)}^*\varphi_{\Psi^X_t(m)},
\end{equation*}
where $P^\gamma$ is the parallel transport with respect to $\nabla^0$ along the integral curve of $X$ starting at $m$.
\end{proposition}

\begin{proof}
Starting from the right-hand side, we have
\begin{align*}
\lim_{t\to 0}&\frac{1}{t}(P^\gamma_{2t,0}(\Psi^X_{-t})_{\Psi^X_{2t}(m)}^*\varphi_{\Psi^X_t(m)}-\varphi_m)\\
&=\lim_{t\to 0}\frac{1}{t}(P^\gamma_{2t,0}(\Psi^X_{-t})_{\Psi^X_{2t}(m)}^*\varphi_{\Psi^X_t(m)}-P^\gamma_{2t,0}\varphi_{\Psi^X_{2t}(m)}+P^\gamma_{2t,0}\varphi_{\Psi^X_{2t}(m)}-\varphi_m)\\
&=2\lim_{t\to 0}\frac{1}{2t}(P^\gamma_{2t,0}\varphi_{\Psi^X_{2t}(m)}-\varphi_m)-\lim_{t\to 0}\frac{1}{t}P^\gamma_{2t,0}(\varphi_{\Psi^X_{2t}(m)}-(\Psi^X_{-t})_{\Psi^X_{2t}(m)}^*\varphi_{\Psi^X_t(m)}).
\end{align*}
From \eqref{covflow}, the first term is $2\,(\nabla^0_X\varphi)_m$. We focus on the second term, by additivity of the flow it equals
\begin{align*}
&\lim_{t\to 0}\frac{1}{t}P^\gamma_{2t,0}(\Psi^X_{-2t})^*_{\Psi^X_{2t}(m)}((\Psi^X_{2t})^*_{m}\varphi_{\Psi^X_{2t}(m)}-(\Psi^X_t)^*_{m}\varphi_{\Psi^X_t(m)})\\
&=\lim_{t\to 0}P^\gamma_{2t,0}(\Psi^X_{-2t})^*_{\Psi^X_{2t}(m)}\left(2\frac{1}{2t}((\Psi^X_{2t})^*_{m}\varphi_{\Psi^X_{2t}(m)}-\varphi_m)-\frac{1}{t}((\Psi^X_t)^*_{m}\varphi_{\Psi^X_t(m)}-\varphi_m)\right).
\end{align*}
That is, by the continuity and since $P^\gamma_{0,0}=\id=(\Psi^X_0)^*_{\Psi^X_0(m)}$,
\begin{equation*}
2\lim_{t\to 0}\frac{1}{2t}((\Psi^X_{2t})^*_{m}\varphi_{\Psi^X_{2t}(m)}-\varphi_m)-\lim_{t\to 0}\frac{1}{t}((\Psi^X_t)^*_{m}\varphi_{\Psi^X_t(m)}-\varphi_m).
\end{equation*}
Finally, using \eqref{Lieflow}, we have that the second term is
\begin{equation*}
2(\pounds_X\varphi)_m-(\pounds_X\varphi)_m=(\pounds_X\varphi)_m.
\end{equation*}
Adding the first and second terms and using \eqref{covLLeq} we get
\begin{equation*}
\lim_{t\to 0}\frac{1}{t}(P^\gamma_{2t,0}(\Psi^X_{-t})_{\Psi^X_{2t}(m)}^*\varphi_{\Psi^X_t(m)}-\varphi_m)=2(\nabla^0_X\varphi)_m-(\pounds_X\varphi)_m=(\pounds_X^s\varphi)_m.
\end{equation*}
\end{proof}

\begin{definition}\label{def: par-flow}
 Let $\nabla$ be a connection on $M$. For $X\in\mathfrak{X}(M)$ and $(t,m)\in \R\times M$ such that $\Psi^X_{2t}(m)$ is defined, we introduce the \textbf{parallel flow} along $X$ from time $t$ to $m$ as the map $\Theta^X_{t,m}\colon T_{\Psi^X_t(m)}M\rightarrow T_mM$, given by
 \begin{equation*}
 \Theta^X_{t,m}\coloneqq P^\gamma_{2t,0}\circ (\Psi^X_{-t})^*_{\Psi^X_{2t}(m)},
 \end{equation*}
 where $P^\gamma$ is the parallel transport with respect to $\nabla^0$ along the integral curve of $X$ starting at $m$.
\end{definition}

We can extend the definition of the parallel flow $\Theta^X_{t,m}$ to tensors of arbitrary type: it first moves the tensor forwards along an integral curve of a vector field $X$ from $\Psi^X_t(m)$ to $\Psi^X_{2t}(m)$ via the flow, and then it takes it backwards to $m$ by the parallel transport with respect to $\nabla^0$, see Figure \ref{UpsT}.

\begin{figure}[h]
\centering
\includegraphics[scale=1.25]{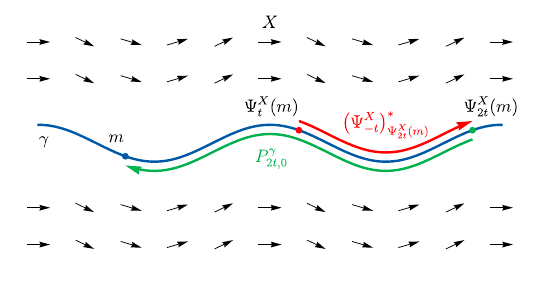}
\caption{Parallel flow.}\label{UpsT}
\end{figure}

\begin{remark}\label{remark: covLL}
 The interpretation in terms of the parallel flow of Proposition \ref{symLflow} makes it possible to define the symmetric Lie derivative for arbitrary tensor fields. Formula \eqref{covLLeq} is then satisfied in general.
\end{remark}

\subsection{The symmetric bracket}\label{sec:symmetric-bracket} By analogy with the formula $\iota_{[X,Y]}=[L_X,\iota_Y]_g$, the corresponding bracket $[\,\,,\,]_s$ in the symmetric setting should be determined by
\begin{equation}\label{eq:symmetric-as-derived-bracket}
\iota_{[X,Y]_s}=[L^s_X,\iota_Y]=\pounds_X^s\circ\iota_Y-\iota_Y\circ\pounds_X^s.\end{equation}
Since the right-hand side is a derivation of degree $(-1)$ and, by Lemma \ref{lem:-1-der-are-contractions}, any such derivation is the interior multiplication by a unique vector field, equation \eqref{eq:symmetric-as-derived-bracket} determines uniquely the symmetric bracket $[\,\,,\,]_s$. We can find an explicit expression as, for $\alpha\in\Upsilon^1(M)$,
\begin{align*}
\pounds^s_X\iota_Y\alpha-\iota_Y\pounds_X^s\alpha & =X\alpha(Y)-(\nabla_X\alpha)(Y)+\alpha(\nabla_YX)\\ & =\alpha(\nabla_XY+\nabla_YX)=\iota_{(\nabla_XY+\nabla_YX)}\alpha.
\end{align*}

\begin{definition}\label{def:sym-bracket} The \textbf{symmetric bracket} corresponding to a connection $\nabla$ on $M$ is given, for $X,Y\in\mathfrak{X}(M)$, by
\begin{equation}\label{SBex}
[X,Y]_s=\nabla_XY+\nabla_YX.
\end{equation}
\end{definition}

\begin{remark}
 To the best of our knowledge, the symmetric bracket was first introduced in \cite{CroGSST}, where the notation
 \begin{equation*}
 \la X:Y\ra=\nabla_XY+\nabla_YX
 \end{equation*}
 was used. The same notation was later adopted by A. Lewis in \cite{LewGISP,Lewbook,LewACD}.
\end{remark}

\begin{example}
Let $g$ be a (pseudo-)Riemannian metric. We have the following formula for the symmetric bracket corresponding to the Levi-Civita connection $\lcn{g}$:
\begin{equation}\label{eq:sym-bracket-LC}
[X,Y]_s=g^{-1}(L_X g(Y)+L_Yg(X))-\grad_g g(X,Y),
\end{equation}
where $\grad_g: \cCi(M)\rightarrow \mathfrak{X}(M)$ is the natural map given by $\grad_g\coloneqq g^{-1}\circ\dif$. 
\end{example}

\begin{remark}
Formula \eqref{eq:sym-bracket-LC} can be reinterpreted, by using $\vartheta\coloneqq g^{-1}$, as a symmetric bracket on the cotangent bundle: for $\alpha,\beta\in\Omega^1(M)$,
\begin{equation}\label{eq: symLA bracket}
[\alpha,\beta]_\vartheta\coloneqq L_{\vartheta(\alpha)}\beta+L_{\vartheta(\beta)}\alpha-\dif\vartheta(\alpha,\beta).
\end{equation}
The bracket in \eqref{eq: symLA bracket} can be defined for an arbitrary $\vartheta\in\mathfrak{X}^2_\text{sym}(M)$, possibly degenerate. It turns out to be a symmetric analogue of the Lie algebroid bracket on the cotangent bundle of a Poisson manifold $(M,\pi)$ (see, for instance, \cite{CouTLA}):
\begin{equation*}
[\alpha,\beta]_\pi\coloneqq L_{\pi(\alpha)}\beta-L_{\pi(\beta)}\alpha-\dif\pi(\alpha,\beta).
\end{equation*}
\end{remark}

Analogues of the usual identities for the Lie bracket, the exterior and Lie derivatives also hold here. On the one hand,
\begin{equation}\label{eq:sym-bracket-as-Lie-derivative}
\pounds_X^sY=[X,Y]_s=\pounds_Y^sX,
\end{equation}
which follows by applying $L_XY=\nabla^0_X Y - \nabla^0_Y X$ and \eqref{covLLeq}. On the other hand, from \eqref{eq:nabla-s} and \eqref{eq:formula-L-nabla-s}, we have, for $\varphi\in\Upsilon^r(M)$ and $X, X_j\in\mathfrak{X}(M)$,
\begin{align*}
(\nabla^s\varphi)(X_1\varlist X_{r+1})=\,&\sum_{j=1}^{r+1}X_j\,\varphi(X_1\varlist \hat{X}_j\varlist X_{r+1})\\
&-\sum_{\underset{i<j}{i,j=1,}}^{r+1}\varphi([X_i,X_j]_s,X_1\varlist \hat{X}_i\varlist \hat{X}_j\varlist X_{r+1}),\\
(\pounds^s_{X}\varphi)(X_1\varlist X_r)=&\,X\,\varphi(X_1\varlist X_k)-\sum_{j=1}^r\varphi([X,X_j]_s,X_1\varlist \hat{X}_j\varlist X_r).
\end{align*}

\subsection{Geometric interpretation of the symmetric bracket}\label{sec:geo-int-sym-bracket}
As a consequence of \eqref{eq:sym-bracket-as-Lie-derivative} and Proposition \ref{symLflow}, the symmetric bracket measures the invariance of a vector field with respect to the parallel flow, that is,
 \begin{equation}\label{eq: sym-br-par-flow}
 \rest{[X,Y]_s}{m}=\rest{\frac{\dif}{\dif t}}{t=0}\Theta^X_{t,m}Y_{\Psi^X_t(m)}
 \end{equation}
A different interpretation, in terms of a second derivative of commutator of flows, is given in \cite[Sec. 3]{LewGISP}.

We recall now an additional geometric interpretation of the symmetric bracket \cite{LewACD}, which describes the properties of involutive distributions.

For a connection $\nabla$ on $M$, a smooth distribution $\Delta\subseteq TM$ is $\nabla$-\textbf{geodesically invariant} if each $\nabla$-geodesic $\gamma \colon I\rightarrow M$ defined on some open interval $I\subseteq\mathbb{R}$ possesses the property that existence of $t_0\in I$ such that $\dot{\gamma}(t_0)\in \Delta_{\gamma(t_0)}$ implies that $\dot{\gamma}(t)\in \Delta_{\gamma(t)}$ for all $t\in I$. A submanifold $S\subseteq M$ is called \textbf{totally} $\nabla$-\textbf{geodesic} if every $\nabla$-geodesic that tangentially intersects $S$ is locally contained in $S$.

\begin{remark}
Note that if a geodesically invariant distribution is, in addition, integrable, it corresponds to a foliation by totally geodesic submanifolds.
\end{remark}

The geometric interpretation for the symmetric bracket is analogous to Frobenius theorem for integrable distributions.

\begin{theorem}[\cite{LewACD}]\label{LewisT}
For a connection $\nabla$ on $M$, a smooth distribution $\Delta\subseteq TM$ with locally constant rank is $\nabla$-geodesically invariant if and only if
\begin{equation*}
[\Gamma(\Delta),\Gamma(\Delta)]_{s}\subseteq \Gamma(\Delta).
\end{equation*}
\end{theorem}

In \cite{Lewbook}, the role of the bracket in theoretical mechanics is discussed.

\section{Commutation relations of symmetric Cartan calculus}\label{sec:commutation-relations}
The classical Cartan calculus on differential forms $\Omega^\bullet(M)$ is characterized by the graded-commutation relations between $\iota$, $\dif$ and $L$:
\begin{align*}
 [\iota_X,\iota_Y]_\text{g}&=0, &[\dif,\dif]_\text{g}&=0, & [\iota_X,\dif]_\text{g}&=L_X,\\
 [L_X,\iota_Y]_\text{g}&=\iota_{[X,Y]}, & [L_X,L_Y]_\text{g}&=L_{[X,Y]}, & [\dif, L_X]_\text{g}&=0.
\end{align*}

Just as graded derivations are closed under the graded commutator, derivations are closed under the commutator. Thus, for symmetric Cartan calculus on $\Upsilon^\bullet(M)$, we aim to have similar relations between $\iota$, $\nabla^s$ and $L^s$ and the usual commutator.

As $\Upsilon^\bullet(M)$ is the space of symmetric forms on $M$ we have
\begin{equation*}
 [\iota_X,\iota_Y]=0.
\end{equation*}
The analogue of the second relation
\begin{equation*}
 [\nabla^s, \nabla^s]=0
\end{equation*}
is also satisfied trivially because of the skew-symmetry of the commutator. The analogues of the third and fourth formulas were used to define the symmetric Lie derivative (Definition \ref{symLalg}) and derive the symmetric bracket \eqref{eq:symmetric-as-derived-bracket}, that is,
\begin{align*}
 &[\iota_X,\nabla^s]=L^s_X, & &[L^s_X,\iota_Y]=\iota_{[X,Y]_s}.
\end{align*}
For the remaining two identities, we will find that the analogues are more involved. In the case of $[L_X,L_Y]_\text{g}=L_{[X,Y]}$, it is clear that $[L^s_X,L^s_Y]\neq L^s_{[ X,Y]_s}$ since the left-hand side is skew-symmetric on $X$ and $Y$, whereas the right-hand side is symmetric. Their proof will require the extension of the contraction operator and the symmetric derivative (Sections \ref{sec:linear-algebra-2} and \ref{sec:extension-sym-derivative}).

\subsection{The symmetric contraction operator}\label{sec:linear-algebra-2}

Since this is a linear-algebraic operator, we first work on a vector space $V$ and then globalize it. Consider $\sigma\in\Sym^kV^*\otimes V$. The \textbf{symmetric contraction operator} \[\iota^s_\sigma\colon\Sym^\bullet V^*\rightarrow\Sym^\bullet V^*\] is the unique degree-$(k-1)$ derivation of $\Sym^\bullet V^*$ determined by its action on $\lambda\in\R$ and $\alpha\in V^*$ by
\begin{align*}
 &\iota^s_\sigma\lambda\coloneqq 0, & &\iota^s_\sigma \alpha\coloneqq \iota_\sigma \alpha,
\end{align*}
where by $\iota_\sigma$ we mean, for $\varphi\in\Sym^rV^*$,
\begin{equation*}
 (\iota_\sigma\varphi)(X_1\varlist X_{r+k-1})\coloneqq \varphi(\sigma(X_1\varlist X_k),X_{k+1}\varlist X_{r+k-1}).
\end{equation*}

We easily see that a generalization of Lemma \ref{lem:-1-der-are-contractions} is true. Namely, the map $\iota^s\colon \sigma\mapsto\iota^s_\sigma$ gives an isomorphism of graded vector spaces:
 \begin{equation*}
\Sym^\bullet V^*\otimes V\overset{\sim}{\longleftrightarrow}\left\{\begin{array}{c}
 \text{derivations of $\Sym^\bullet V^*$ that}\\
 \text{annihilate scalars $\Sym^0V^*=\R$}
 \end{array}\right\}.
\end{equation*}

\begin{lemma}\label{lem: symcont}
 Let $\varphi\in\Sym^r V^*$ and $\sigma\in\Sym^kV^*\otimes V$. The symmetric contraction operator is explicitly given by
\begin{equation*}
 \iota^s_\sigma\varphi=\frac{(r+k-1)!}{(r-1)!k!}\sym\iota_\sigma\varphi.
\end{equation*}
\end{lemma}

\begin{proof}
The endomorphism on the right-hand side is clearly a degree-$(k-1)$ map and coincides with $\iota^s_\sigma$ on scalars and $1$-forms. Therefore, it is enough to check that it is a derivation of $\Sym^\bullet V^*$, which is done analogously to Lemma \ref{lem:-1-der-are-contractions}. \end{proof}

Globally, the symmetric contraction operator by an element of $\Upsilon^k(M, TM)$ (in the sense of Remark \ref{rem:notation-upsilon-bundle}) becomes an element of $\der_{k-1}(\Upsilon^\bullet(M))$.

\subsection{Commutator of symmetric Lie derivatives}

We can now deal with the analogue of the fifth identity.

\begin{proposition}\label{prop: com1}
 Let $\nabla$ be a torsion-free connection on $M$. For $X, Y\in\mathfrak{X}(M)$,
 \begin{equation}\label{eq: com1}
 [L^s_X, L^s_Y]=L^s_{[X,Y]}+\iota^s_{2(\nabla_{\nabla X}Y-\nabla_{\nabla Y}X-R_\nabla (X,Y))},
 \end{equation}
 where $R_\nabla $ is the Riemann curvature tensor.
\end{proposition}

\begin{proof}
 As all the terms are derivations of $\Upsilon^\bullet (M)$, it is enough to check the equality on functions and $1$-forms. Functions are annihilated by any symmetric contraction operator, so the relation is clearly satisfied as $L^s_Xf=Xf$. For an arbitrary $1$-form $\alpha\in\Upsilon^1(M)$ we have
 \begin{align*}
 ([L^s_X,L^s_Y&]\alpha)(Z)-(L^s_{[X,Y]}\alpha)(Z)\\
 =\,&X(L^s_Y\alpha)(Z)-(L^s_Y\alpha)(\pg{X,Z}_s)-Y(L^s_X\alpha)(Z)+(L^s_X\alpha)(\pg{Y,Z}_s)\\
 &-[X,Y]\alpha(Z)+\alpha(\pg{[X,Y],Z}_s)\\
 =\,&X(Y\alpha(Z))-X\alpha(\pg{Y,Z}_s)-Y\alpha(\pg{X,Z}_s)+\alpha(\pg{Y,\pg{X,Z}_s}_s)\\
 &-Y(X\alpha(Z))+Y\alpha(\pg{X,Z}_s)+X\alpha(\pg{Y,Z}_s)-\alpha(\pg{X,\pg{Y,Z}_s}_s)\\
 &-[X,Y]\alpha(Z)+\alpha(\pg{[X,Y],Z}_s)\\
 =\,&\alpha(W),
 \end{align*}
 where $W\in\mathfrak{X}(M)$ is given by $W\coloneqq \pg{\pg{X,Z}_s,Y}_s-\pg{X,\pg{Y,Z}_s}_s+\pg{[X,Y],Z}_s$.
 Explicitly,
 \begin{align*}
W=&\nabla_Y\nabla_XZ+\nabla_Y\nabla_ZX+\nabla_{\nabla_XZ+\nabla_ZX}Y-\nabla_X\nabla_YZ\\& -\nabla_X\nabla_ZY -\nabla_{\nabla_YZ+\nabla_ZY}X
 +\nabla_{[X,Y]}Z+\nabla_Z[X,Y].
 \end{align*}
 Using the definition of the curvature tensor and the torsion-freeness of $\nabla$, we get
 \begin{equation*}
 W=-R_\nabla (X,Y)Z+R_\nabla (Y,Z)X+R_\nabla (Z,X)Y+2(\nabla_{\nabla_ZX}Y-\nabla_{\nabla_ZY}X).
 \end{equation*}
 Finally, it follows from the algebraic Bianchi identity that
 \begin{equation*}
 ([L^s_X,L^s_Y]\alpha)(Z)-(L^s_{[X,Y]}\alpha)(Z)=\alpha(2(\nabla_{\nabla_ZX}Y-\nabla_{\nabla_ZY}X -R_\nabla (X,Y)Z)).
 \end{equation*}
\end{proof}

\subsection{Extension of the symmetric derivative and the symmetric curvature}\label{sec:extension-sym-derivative}

Similarly as we generalized the contraction operator, we generalize also the symmetric derivative.

\begin{definition}
Let $\nabla$ be a connection on $M$ and $A\in\Gamma(\en TM)$. We introduce the $A$-\textbf{symmetric derivative} of $\nabla^s_A\in\der_1(\Upsilon^\bullet(M))$ by the formula
\begin{equation*}
 \nabla^s_A\coloneqq [\iota^s_A,\nabla^s].
 \end{equation*}
\end{definition}

\begin{remark}\label{rem: nablas_id}
 Note that $\iota^s_{\id}\varphi=r\,\varphi$ for every $\varphi\in \Upsilon^r(M)$. Therefore, $\nabla^s=\nabla^s_{\id_{TM}}$.
\end{remark}


The last piece to be introduced is a symmetric analogue of the curvature operator, whose definition will rely on the concept of \textbf{second covariant derivative}. This is, for a connection $\nabla$ on $M$, the linear map $\nabla^2\colon \mathfrak{X}(M)\rightarrow\Gamma(\otimes^2T^*M\otimes TM)$ defined by
\begin{equation*}
 (\nabla^2X)(Y,Z)\coloneqq (\nabla\nabla X)(Y,Z)=\nabla_Y\nabla_ZX-\nabla_{\nabla_YZ}X.
\end{equation*}
Note that the map is not $\cCi(M)$-linear. However, when we take its skew-symmetric part and the connection is torsion-free, it is and we recover the Riemann curvature tensor $R_\nabla \in\Omega^2(M,\en TM)$:
\begin{equation*}
 (2\ske\nabla^2X)(Y,Z)=\nabla_Y\nabla_ZX-\nabla_Z\nabla_YX-\nabla_{[Y,Z]}X=R_\nabla (Y,Z)X.
\end{equation*}

\begin{definition}\label{def: sym-curv}
 Let $\nabla$ be a connection on $M$. We introduce the \textbf{symmetric curvature operator} $R^s_\nabla \colon \mathfrak{X}(M)\rightarrow\Upsilon^2(M, TM)$ by the formula
 \begin{equation*}
 (R^s_\nabla X)(Y,Z)\coloneqq (2\sym \nabla^2X)(Y,Z)= \nabla_Y\nabla_ZX+\nabla_Z\nabla_YX-\nabla_{\pg{Y,Z}_s}X.
 \end{equation*}
\end{definition}

\begin{remark}\label{rem: curvature and sym-curvature}
 If $\nabla$ is a torsion-free connection on $M$, we can express the second covariant derivative in terms of the corresponding Riemannian curvature tensor and symmetric curvatures operator:
 \begin{equation*}
 (\nabla^2X)(Y,Z)=\frac{1}{2}(R_\nabla (Y,Z)X+(R^s_\nabla X)(Y,Z)).
 \end{equation*}
\end{remark}

A geometric interpretation of the symmetric curvature operator will be given in Remark \ref{rem:geo-meaning-sym-curv}.

\subsection{Commutator of symmetric and symmetric Lie derivatives}

We prove now the remaining identity of symmetric Cartan calculus.

\begin{proposition}\label{prop:commutator-sym-der-sym-Lie}
 Let $\nabla$ be a torsion-free connection. For $X\in\mathfrak{X}(M)$,
 \begin{equation}\label{eq: com2}
 [\nabla^s,L^s_X]=2\nabla^s_{\nabla X}+\iota^s_{2\sym \iota_XR_\nabla +R^s_\nabla X}.
 \end{equation}
\end{proposition}

\begin{proof}
 Analogously as in Proposition \ref{prop: com1}, it is enough to check the relation on $\cCi(M)$ and $\Upsilon^1(M)$. The relation \eqref{eq: com2} restricted on functions says
 \begin{equation*}
 2\dif\circ\iota_X\circ\dif-\iota_X\circ\nabla^s\circ\dif=2\iota_{\nabla X}\circ\dif.
 \end{equation*}
 The left-hand side acting on $f\in \cCi(M)$ and evaluated on $Y\in\mathfrak{X}(M)$ gives
 \begin{align*}
 2Y(Xf)-(\nabla^s\dif f)(X,Y)
 &=Y(Xf)-X(Yf)+[X,Y]_sf\\ &=([X,Y]+[X,Y]_s)f
 =2(\nabla_YX)f=2(\iota_{\nabla X}\dif f)(Y),
 \end{align*}
 where the second equality follows from torsion-freeness of $\nabla$. For $\alpha\in\Upsilon^1(M)$,
 \begin{align*}
 ([\nabla^s&,L^s_X]\alpha)(Y,Y)-2(\nabla^s_{\nabla X}\alpha)(Y,Y)\\
=\,&(\nabla^sL^s_X\alpha)(Y,Y)-(L^s_X\nabla^s\alpha)(Y,Y)-2(\iota^s_{\nabla X}\nabla^s\alpha)(Y,Y)+2(\nabla^s\iota_{\nabla X}\alpha)(Y,Y)\\
 =\,&2Y(L^s_X\alpha)(Y)-(L^s_X\alpha)(\pg{Y,Y}_s)-X(\nabla^s\alpha)(Y,Y)+2(\nabla^s\alpha)([X,Y]_s,Y)\\
 &-4(\nabla^s\alpha)(\nabla_YX,Y)+4Y\alpha(\nabla_YX)-2\alpha(\nabla_{\pg{Y,Y}_s}X)\\
 =\,&2Y(X\alpha(Y))-2Y\alpha([X,Y]_s)-X\alpha(\pg{Y,Y}_s)+\alpha(\pg{X,\pg{Y,Y}_s}_s)\\
 &-2X(Y\alpha(Y))+X\alpha(\pg{Y,Y}_s)+2[X,Y]_s\alpha(Y)+2Y\alpha([X,Y]_s)\\
 &-2\alpha(\pg{[X,Y]_s,Y}_s)-4(\nabla_YX)\alpha(Y)-4Y\alpha(\nabla_YX)+4\alpha(\pg{\nabla_YX,Y}_s)\\
 &+4Y\alpha(\nabla_YX)-2\alpha(\nabla_{\pg{Y,Y}_s}X).
 \end{align*}
 Using the fact that $\nabla$ is torsion-free we get 
\begin{equation*}
 ([\nabla^s,L^s_X]\alpha)(Y,Y)-2(\nabla^s_{\nabla X}\alpha)(Y,Y)=\alpha(W),
\end{equation*}
where $W\in\mathfrak{X}(M)$ is given by
\begin{align*}
 W&= \pg{X,\pg{Y,Y}_s}_s-2\pg{[X,Y]_s,Y}_s+4\pg{\nabla_YX,Y}_s-2\nabla_{\pg{Y,Y}_s}X)\\
 &=2\nabla_X\nabla_YY+2\nabla_{\nabla_YY}X-2\nabla_{\nabla_XY+\nabla_YX}Y-2\nabla_Y\nabla_XY
 \\ & \phantom{=} -2\nabla_Y\nabla_YX
+4\nabla_{\nabla_YX}Y+4\nabla_Y\nabla_YX-4\nabla_{\nabla_YY}X.
\end{align*}
 It follows from the torsion-freeness of $\nabla$ that $W$ may be rewritten in terms of the Riemann curvature tensor and the symmetric curvature operator:
 \begin{equation*}
 W=2R_\nabla (X,Y)Y+(R^s_\nabla X)(Y,Y)=(2\sym\iota_X R_\nabla +R^s_\nabla X)(Y,Y).
 \end{equation*}
 Finally, it follows from polarization that
 \begin{equation*}
 ([\nabla^s,L^s_X]\alpha)(Y,Z)-2(\nabla^s_{\nabla X}\alpha)(Y,Z)=\alpha((2\sym\iota_X R_\nabla +R^s_\nabla X)(Y,Z)).
 \end{equation*}
\end{proof}

Identities \eqref{eq: com1} and \eqref{eq: com2} involve rather complicated terms. We will be able to say more about them once we have described affine morphisms in Section \ref{sec:affine-manifold-morphism}.

\subsection{Dependence on the connection}

The symmetric contraction operator introduced in Section \ref{sec:linear-algebra-2} allows us to describe how symmetric Cartan calculus varies when we change the connection.

\begin{proposition}\label{prop: symCartan variation}
Let $\nabla$ and $\nabla'$ be two torsion-free connections on $M$, then
\begin{equation*}
 \nabla'_XY=\nabla_XY-\frac{1}{2}\sigma(X,Y)
 \end{equation*}
 for a unique $\sigma\in\Upsilon^2(M,TM)$. The corresponding symmetric derivatives, symmetric Lie derivatives and symmetric brackets are related as follows:
 \begin{align*}
 &\nabla'^s=\nabla^s+\iota^s_\sigma, & &L'^s_X=L^s_X+\iota^s_{\sigma(X,\,\,)}, & &\pg{X,Y}'_s=[X,Y]_s-\sigma(X,Y).
 \end{align*} 
\end{proposition}

\begin{proof}
 The statement for symmetric brackets is trivial. It is enough to check the other two relations on functions and $1$-forms. The equalities on functions follow from that every symmetric contraction operator annihilate functions, every symmetric derivative is geometric and $L^s_Xf=Xf=L'^s_Xf$. The equalities on $1$-forms follow by straightforward calculations provided the relation between symmetric brackets.
\end{proof}


\section{Affine manifold morphisms}\label{sec:affine-manifold-morphism}

We study now affine morphisms, which are the natural transformations in presence of a connection. We will see how they interact with symmetric Cartan calculus and also their infinitesimal version. Some of the results of this section are stated for arbitrary connections since they provide stronger versions.

\subsection{Definition of affine morphisms}
Given two manifolds $M$ and $M'$, any diffeomorphism $\phi\colon M\rightarrow M'$ intertwines their Cartan calculi in the sense that,
\begin{align*}
 \phi_*[X,Y]&=[\phi_*X,\phi_*Y],& \phi^*\circ\dif&=\dif\circ\phi^*,& \phi^*\circ\pounds_{\phi_*X}&=\pounds_X\circ\phi^*.
\end{align*}

For symmetric Cartan calculus, the notion of affine diffeomorphism will play an analogous role. In this section, connections are not necessarily torsion-free.
\begin{definition}\label{def: affine}
Given two manifolds with connection $(M,\nabla)$ and $(M',\nabla')$, a smooth map $\phi\colon M\rightarrow M'$ is an \textbf{affine morphism} if $\phi_*$ commutes with the corresponding parallel transports $P$ and $P'$, that is,
\begin{equation*}
\phi_{*\gamma(t_1)}\circ P^\gamma_{t_0,t_1}=P'^{\phi\circ\gamma}_{t_0,t_1}\circ \phi_{*\gamma(t_0)},
\end{equation*}
for any curve $\gamma\colon I\rightarrow M$ and $t_0, t_1\in I$. If, in addition, $\phi$ is a diffeomorphism, we call it \textbf{affine diffeomorphism}. We denote the space of affine diffeomorphisms by $\Aff(M,\nabla)$ if $M=M'$ and $\nabla=\nabla'$.
\end{definition}

It is well known that the group of isometries of a (pseudo-)Riemannian metric is a Lie group of dimension at most $\frac{1}{2}n(n+1)$, where $n\coloneqq \dim M$. A similar result is true also for the group $\Aff(M,\nabla)$.

\begin{theorem}[\cite{KobTGDG}]
 Let $\nabla$ be a connection on $M$ of dimension $n\in\mathbb{N}$. The group $\Aff(M,\nabla)$ is a Lie group of dimension at most $n(n+1)$.
\end{theorem}

\begin{example}\label{ex: affine group}
 In the simplest case, the Euclidean connection $\nabla^\emph{Euc}$ on $\mathbb{R}^n$, we have $\Aff(\R^n,\nabla^\emph{Euc})\cong \Aff(n,\R)\coloneqq \GL(n,\R)\ltimes\R^n$, thus attaining the maximal dimension.
\end{example}

\begin{example} 
Every isometry between two (pseudo-)Riemannian metrics is an affine diffeomorphism between the corresponding Levi-Civita connections, see e.g. \cite{KobTGDG}. However, the converse is not true in general as it is clear from Example \ref{ex: affine group}.
\end{example}

A simpler characterization of affine morphisms will prove useful. 

\begin{proposition}\label{prop: affine-morphism}
 A map $\phi\colon M\rightarrow M'$ is an affine morphism if and only if
\begin{equation*}\label{eq: affmorph}
\nabla\circ\phi^*=\phi^*\circ\nabla',
 \end{equation*}
 where $\nabla$ is seen as a map $\Omega^1(M)\rightarrow\Omega^1(M, T^*M)$.
\end{proposition}

\begin{proof}
 Recall that if $P^\gamma_{t_0,t_1}$ is the parallel transport with respect to $\nabla$ from $t_0$ to $t_1$ along $\gamma$, its algebraic transpose is the parallel transport from $t_1$ to $t_0$ along $\gamma$ for the dual connection on $T^*M$, which we denote by ${P^*}^\gamma_{t_1,t_0}$. Therefore, a map $\phi\colon M\rightarrow M'$ is an affine morphism if and only if
 \begin{equation}\label{eq: aff morph proof}
{P^*}^\gamma_{t_0,t_1}\circ\phi^*_{\gamma(t_0)}=\phi^*_{\gamma(t_1)}\circ {P'^*}^{\phi\circ\gamma}_{t_0,t_1}.
 \end{equation}
 Consider an arbitrary vector $u\in TM$ and $\alpha\in\Omega^1(M')$. We have,
 \begin{equation*}
 \nabla_u\phi^*\alpha=\lim_{t\to 0}\frac{1}{t}({P^*}^\gamma_{t,0}\phi^*_{\gamma(t)}\alpha_{\phi(\gamma(t))}-\phi^*_{\gamma(0)}\alpha_{\phi(\gamma(0))}),
 \end{equation*}
 where $\gamma\colon I\rightarrow M$ is an arbitrary curve such that $\dot{\gamma}(0)=u$. Assume first that $\phi$ is an affine morphism. It follows from \eqref{eq: aff morph proof} that
 \begin{align*}
 \iota_u(\nabla\phi^*\alpha)&=\nabla_u\phi^*\alpha=\lim_{t\to 0}\frac{1}{t}(\phi^*_{\gamma(0)}{P'^*}^{\phi\circ\gamma}_{t,0}\alpha_{\phi(\gamma(t))}-\phi^*_{\gamma(0)}\alpha_{\phi(\gamma(0))})\\
 &=\phi^*_{\gamma(0)}\lim_{t\to 0}\frac{1}{t}({P'^*}^{\phi\circ\gamma}_{t,0}\alpha_{\phi(\gamma(t))}-\alpha_{\phi(\gamma(0))})=\phi^*_{\gamma(0)}\nabla'_{\phi_{*\gamma(0)}u}\alpha=\iota_u(\phi^*\nabla'\alpha).
 \end{align*}
 Conversely, assume that $\nabla\circ\phi^*=\phi^*\circ \nabla'$ is true. Consider an arbitrary curve $\gamma\colon I\rightarrow M$, $t_0, t_1\in I$ and $\zeta\in T^*_{\phi(\gamma(t_0))} M'$. To prove that $\phi$ is an affine morphism, by \eqref{eq: aff morph proof}, it is enough to show that the $1$-form $\alpha\in\Gamma(\gamma^*TM)$ defined, for $t\in I$, by
 \begin{align*}
 \alpha(t)&\coloneqq \phi^*_{\gamma(t)}\beta(t), & \beta(t)&\coloneqq {P'^*}^{\phi\circ\gamma}_{t_0,t}\zeta,
 \end{align*}
 satisfies $\nabla_{\dot{\gamma}}\alpha=0$. This follows from $\nabla'_{\phi_*\dot{\gamma}}\beta=0$, as
 \begin{equation*}
\nabla_{\dot{\gamma}}\alpha=\nabla_{\dot{\gamma}}\phi^*\beta=\phi^*\nabla'_{\phi_*\dot{\gamma}}\beta=0.
 \end{equation*}
\end{proof}

For affine diffeomorphisms, we can reformulate Proposition \ref{prop: affine-morphism} in terms of the original connection on $TM$.

\begin{proposition}[e.g., \cite{KoNoFDG}]\label{prop: affnabmorph}
A diffeomorphism $\phi\colon M\rightarrow M'$ is an affine diffeomorphism if and only if, for all $X, Y\in\mathfrak{X}(M)$,
\begin{equation*}
\phi_*(\nabla_XY)=\nabla'_{\phi_*X}\phi_*Y
\end{equation*}
\end{proposition}


Given a diffeomorphism $\phi\colon M\rightarrow M'$ and a connection $\nabla'$ on $M'$, the \textbf{pullback connection} $\phi^*\nabla'$ on $M$ is given for $X,Y\in\mathfrak{X}(M)$ by
\begin{equation*}
(\phi^*\nabla')_XY\coloneqq \phi^{-1}_*\nabla'_{\phi_*X}\phi_*Y.
\end{equation*}

\begin{corollary}\label{cor: pullback con}
 Let $\phi\colon M\rightarrow M'$ be a diffeomorphism and $\nabla'$ be a connection on $M'$. Then $\phi$ is an affine diffeomorphism between $(M,\phi^*\nabla')$ and $(M',\nabla')$.
\end{corollary}

\begin{remark}\label{rk: pullbakc torsion}
 Since $\phi_*$ preserves the Lie bracket of vector fields, we get
\begin{equation*}
 T_{\phi^*\nabla'}(X,Y)=\phi^{-1}_*T_{\nabla'}(\phi_*X,\phi_*Y).
\end{equation*}
Therefore, $\phi^*\nabla'$ is torsion-free if and only if $\nabla'$ is torsion-free.
\end{remark}

\subsection{Relation to symmetric Cartan calculus}

In case of torsion-free connections, we can rephrase Proposition \ref{prop: affine-morphism} using the symmetric derivatives.

\begin{corollary}\label{cor: affmorph}
 Let $\nabla$ and $\nabla'$ be torsion-free connections on $M$ and $M'$ respectively. A smooth map $\phi\colon M\rightarrow M'$ is an affine morphism if and only if
 \begin{equation*}
\nabla^s\circ\phi^*=\phi^*\circ{\nabla'}^s.
 \end{equation*}
\end{corollary}

\begin{remark}\label{rk: aff morph and nablas}
 Note that if $\nabla$ and $\nabla'$ are not torsion-free, we can only claim that if $\phi\colon M\rightarrow M'$ is an affine morphism, then $\nabla^s \circ\phi^*=\phi^*\circ \nabla'^s$.
\end{remark}

The fact that the pullback of a smooth map intertwines two symmetric derivatives has a very clear geometrical interpretation. 

\begin{proposition}\label{prop:affgeo}
 Let $\nabla$ and $\nabla'$ be connections on $M$ and $M'$ respectively. A smooth map $\phi\colon M\rightarrow M'$ satisfies $\nabla^s\circ \phi^*=\phi^*\circ\nabla'^s$ if and only if $\phi$ is geodesic preserving, that is, $\gamma$ being a $\nabla$-geodesic implies that $\phi\circ\gamma$ is a $\nabla'$-geodesic.
\end{proposition}

\begin{proof}
Consider an arbitrary curve $\gamma\colon I\rightarrow M$ and denote the curve $\phi\circ\gamma$ on $M'$ by $\gamma'$. By a straightforward calculation, for every $\alpha\in\Upsilon^1(M')$ and $t\in I$, one finds 
\begin{equation}\label{eq:affgeo}
 (\nabla^s\phi^*\alpha)(\dot{\gamma}(t),\dot{\gamma}(t))-(\phi^*\nabla'^s\alpha)(\dot{\gamma}(t),\dot{\gamma}(t))=2((\phi^*\alpha)(\nabla_{\dot{\gamma}}\dot{\gamma})-\alpha(\nabla'_{\dot{\gamma}'}\dot{\gamma}'))(t).
\end{equation}
If the pullback $\phi^*$ intertwines the symmetric derivatives, we get
\begin{equation*}
 (\phi^*\alpha)(\nabla_{\dot{\gamma}}\dot{\gamma})=\alpha(\nabla'_{\dot{\gamma}'}\dot{\gamma}').
\end{equation*}
Clearly, $\gamma$ being a $\nabla$-geodesic implies that $\gamma'=\phi\circ\gamma$ is a $\nabla'$-geodesic.

Conversely, given an arbitrary $u\in TM$, there is a $\nabla$-geodesic $\gamma$ such that $\dot{\gamma}(0)=u$. It follows from \eqref{eq:affgeo} and the fact that $\gamma'$ is, by assumption, a $\nabla'$-geodesic, that
\begin{equation*}
 (\nabla^s\phi^*\alpha)(u,u)=(\phi^*\nabla'^s\alpha)(u,u).
\end{equation*}
Since symmetric derivatives are geometric derivations of $\Upsilon^\bullet(M)$, we obtain, by polarization, that $\nabla^s\circ\phi^*=\phi^*\circ\nabla'^s$.
\end{proof}

The relation between affine diffeomorphisms, geodesic preserving diffeomorphisms, and symmetric Cartan calculus is described by the following result.

\begin{proposition}\label{prop: Affine morphisms and SCC}
Let $\nabla$ and $\nabla'$ be connections on $M$ and $M'$ and denote the associated torsion-free connections by $\nabla^0$ and $\nabla'^0$ respectively. Then the following statements about a~diffeomorphism $\phi\colon M\rightarrow M'$ are equivalent:
\begin{enumerate}[]
\item $\phi$ is an affine diffeomorphism between $(M,\nabla^0)$ and $(M',\nabla'^0)$;
\item $\phi$ is geodesic preserving;
\item $\nabla^s\circ \phi^*=\phi^*\circ\nabla'^s$;
\item $\pounds^s_{X}\circ \phi^*=\phi^*\circ\pounds'^s_{\phi_*X}$ for all $X\in\mathfrak{X}(M)$;
\item $\phi_*[X,Y]_s=\pg{\phi_*X,\phi_*Y}'_s$ for all $X, Y\in\mathfrak{X}(M)$.
\end{enumerate}
\end{proposition}

\begin{proof}
The equivalence of \emph{(1)}, \emph{(2)} and \emph{(3)} is given by Corollary \ref{cor: affmorph} and Proposition \ref{prop:affgeo}. We see that \emph{(3)} implies \emph{(4)} by using the identity $\iota_X\circ\phi^*=\phi^*\circ\iota_{\phi_*X}$:
\begin{align*}
\phi^*\circ\pounds'^s_{\phi_*X}&=\phi^*\circ\iota_{\phi_*X}\circ\nabla'^s-\phi^*\circ\nabla'^s\circ\iota_{\phi_*X}=\iota_X\circ\nabla^s\circ\phi^*-\nabla^s\circ\iota_{X}\circ \phi^*\\
&=\pounds_{X}^{s}\circ \phi^*.
\end{align*}
To see that \emph{(4)} implies \emph{(5)}, simply consider \emph{(4)} for any $\alpha\in\Upsilon^1(M')$, that is,
\begin{align*}
0=\,&(\phi^*\pounds'^{s}_{\phi_*X}\alpha)(Y)-(\pounds^{s}_{X}\phi^*\alpha)(Y)=\phi^*((\pounds_{\phi_*X}'^s\alpha)(\phi_*Y))-(\pounds^s_{X}\phi^*\alpha)(Y)\\
=\,&\phi^*((\phi_*X)\alpha(\phi_*Y))-\phi^*(\alpha(\pg{\phi_*X,\phi_*Y}'_{s}))-X(\phi^*\alpha)(Y)+(\phi^*\alpha)([X,Y]_s)\\
=\,&\phi^*(\alpha(\pg{\phi_*X,\phi_*Y}'_{s}-\phi_*[X,Y]_s)).
\end{align*}
The equivalence of \emph{(5)} and \emph{(1)} follows from Proposition \ref{prop: affnabmorph} and the identities
\begin{align*}
\phi_*\nabla^0_XY&=\frac{1}{2}(\phi_*[X,Y]_s+\phi_*[X,Y]),\\\nabla'^0_{\phi_*X}\phi_*Y&=\frac{1}{2}(\pg{\phi_*X,\phi_*Y}'_{s}+[\phi_*X,\phi_*Y]).
\end{align*}
\end{proof}

\subsection{Affine vector fields}
The infinitesimal version of an affine diffeomorphism from the manifold to itself is a vector field whose flow is a local $1$-parameter subgroup of affine diffeomorphisms. We call such vector field an \textbf{affine vector field} and denote the vector space of all affine vector fields by $\aff(M,\nabla)$. There is a useful equivalent characterization of affine vector fields.

\begin{proposition}[\cite{KobTGDG}]\label{prop: affine vf}
 Let $\nabla$ be a connection on $M$. A vector field $X\in\mathfrak{X}(M)$ is affine if and only if, for all $Y, Z\in\mathfrak{X}(M)$,
 \begin{equation}\label{eq: affine vector field}
 [L_X,\nabla_Y]Z=\nabla_{[X,Y]}Z.
 \end{equation}
\end{proposition}

\begin{example}\label{ex:Killing-vec-field}
 Consider a (pseudo-)Riemannian manifold $(M,g)$. Every Killing vector field for $g$ is an affine vector field for the corresponding Levi-Civita connection. However, the converse is not true in general. For instance, the affine vector field $X\coloneqq x^i\partial_{x^i}\in\mathfrak{X}(\R^n)$ for the Euclidean connection is not a Killing vector field for the Euclidean metric $g^\emph{Euc}$, since we have $L_Xg^\emph{Euc}=2g^\emph{Euc}\neq 0$.
\end{example}


If the connection is torsion-free, there are three more equivalent formulae characterizing affine vector fields that are closely related to symmetric Cartan calculus and will also be relevant in Section \ref{sec: simp-com}.

\begin{proposition}\label{prop: affine vector field tf}
 Let $\nabla$ be a torsion-free connection on $M$. A vector field $X\in\mathfrak{X}(M)$ is affine if and only if one of the following equivalent statements is true
\begin{enumerate}
 \item $2\sym\iota_X R_\nabla +R^s_\nabla X=0$,
 \item $[X,\pg{Y,Z}_s]=\pg{[X,Y],Z}_s+\pg{Y,[X,Z]}_s$ for all $Y,Z\in\mathfrak{X}(M)$.
 \item $[\nabla^s, L_X]=0$.
\end{enumerate}
\end{proposition}

\begin{proof}
 Using the fact that $\nabla$ is torsion-free, we can rewrite the relation \eqref{eq: affine vector field} as
 \begin{equation*}
 \nabla_X\nabla_YZ-\nabla_{\nabla_YZ}X-\nabla_Y\nabla_XZ+\nabla_Y\nabla_ZX=\nabla_{[X,Y]}Z.
 \end{equation*}
 Taking all the terms on the left-hand side yields $R_\nabla (X,Y)Z+(\nabla^2X)(Y,Z)=0$. This is thanks to Remark \ref{rem: curvature and sym-curvature} equivalent to 
 \begin{equation*}
 2R_\nabla (X,Y)Z+R_\nabla (Y,Z)X+(R^s_\nabla X)(Y,Z)=0.
 \end{equation*}
 Finally, using the algebraic Bianchi identity yields
 \begin{equation*}
 R_\nabla (X,Y)Z+R_\nabla (X,Z)Y+(R^s_\nabla X)(Y,Z)=0.
 \end{equation*} 
On the other hand, by Proposition \ref{covLL}, see also Remark \ref{remark: covLL}, we can replace all the covariant derivatives in \eqref{eq: affine vector field} with the Lie and symmetric Lie derivatives. Therefore, \eqref{eq: affine vector field} is satisfied if and only if
\begin{equation*}
 [L_X,L_Y]Z+[L_X,L^s_Y]Z=L_{[X,Y]}Z+L^s_{[X,Y]}Z.
\end{equation*}
Using the Jacobi identity for the Lie bracket of vector fields yields
\begin{equation*}
 [L_X,L^s_Y]Z=L^s_{[X,Y]}Z,
\end{equation*}
which can be expressed in terms of the Lie and symmetric bracket as
\begin{equation*}
 [X,\pg{Y,Z}_s]-\pg{Y,[X,Z]}_s=\pg{[X,Y],Z}_s.
\end{equation*}
Clearly $\rest{[\nabla^s, L_X]}{\cCi(M)}=\rest{[\dif, L_X]_\text{g}}{\cCi(M)}=0$. Moreover, for the commutator restricted on $1$-forms, we have
\begin{align*}
 \iota_Y&\circ\rest{[\nabla^s, L_X]}{\Upsilon^1(M)}\\
 &= \iota_Y\circ\nabla^s\circ L_X-\iota_Y\circ L_X\circ\nabla^s\\
 &=L^s_Y\circ L_X+\nabla^s\circ\iota_Y\circ L_X+\iota_{[X,Y]}\circ \nabla^s-L_X\circ\iota_Y\circ\nabla^s\\
 &=L^s_Y\circ L_X-L_X\circ L^s_Y+L^s_{[X,Y]}+\nabla^s\circ\iota_{[X,Y]}+\nabla^s\circ\iota_Y\circ L_X-L_X\circ\nabla^s\circ\iota_Y.
\end{align*}
Since $\nabla^s$ is geometric, we can replace it with the exterior derivative $\dif$. Using the identities $[\dif, L_X]_\text{g}=0$ and $[L_X,\iota_Y]_\text{g}=\iota_{[X,Y]}$ yields
\begin{equation*}
 \rest{(\iota_Y\circ[\nabla^s, L_X])}{\Upsilon^1(M)}=L^s_{[X,Y]}-[L_X,L^s_Y].
\end{equation*}
By a straightforward calculation one finds
\begin{equation*}
 ([\nabla^s, L_X]\alpha)(Y,Z)=\alpha([X,\pg{Y,Z}_s]-\pg{Y,[X,Z]}_s-\pg{[X,Y],Z}_s).
\end{equation*}
\end{proof}

\begin{remark}\label{rem:geo-meaning-sym-curv}
If, in particular, $\nabla$ is flat and torsion-free, $X\in\aff(M,\nabla)$ if and only if $R^s_\nabla X=0$, which gives a geometric interpretation to symmetric curvature operator.
\end{remark}

Note that the expression on the left-hand side of \emph{(1)} in Proposition \ref{prop: affine vector field tf} already appeared in Proposition \ref{prop:commutator-sym-der-sym-Lie}. Therefore, we have the following.

\begin{corollary}\label{cor: affine-vf}
 Let $\nabla$ be a torsion-free connection on $M$, we have that $X\in\aff(M,\nabla)$ if and only if
 \begin{equation*}
[\nabla^s,L^s_X]=2\nabla^s_{\nabla X}.
 \end{equation*}
\end{corollary}

\subsection{Relation to parallel flow}
For a given connection $\nabla$, the space of all affine vector fields form a Lie subalgebra of $\mathfrak{X}(M)$. If $\nabla$ is, in addition, torsion-free, we have the natural abelian Lie subalgebra of $\aff(M,\nabla)$:
\begin{equation*}
 \aff_0(M,\nabla)\coloneqq\{X\in\aff(M,\nabla)\,|\, \nabla X=0\}.
\end{equation*}
We follow the standard terminology and refer to the condition $\nabla X=0$ by saying that $X$ is parallel with respect to $\nabla$.

\begin{example}
 For the Euclidean space, we straightforwardly find that
 \begin{align*}
 \aff_0(\R^n,\nabla^\emph{Euc})&=\spann\{\partial_{x^i}\,|\,i\in\{1\varlist n\}\},\\
 \aff(\R^n,\nabla^\emph{Euc})&=\aff_0(\R^n,\nabla^\emph{Euc})\oplus\spann\{x^i\partial_{x^j}\,|\,i,j\in\{1\varlist n\}\}.
 \end{align*}
\end{example}

The next proposition provides a characterization of $\aff_0(M,\nabla)$ in terms of parallel flow (Definition \ref{def: par-flow}).

\begin{proposition}\label{prop: parallel-flow}
 \sloppy For a torsion-free connection $\nabla$, we have that $X\in\mathrm{aff}_0(M,\nabla)$ if and only if $X\in\aff(M,\nabla)$ and the parallel flow of $X$ is equal to its actual flow. By the equality, we mean precisely that, for every $(t,m)\in\R\times M$ such that $\Psi^X_{2t}(m)$ is defined, we have
 \begin{equation*}
 \Theta^X_{t,m}=(\Psi^X_{-t})_{*\Psi^X_{t}(m)}.
 \end{equation*}
\end{proposition}

\begin{proof}
 We start by assuming that $X\in\mathrm{aff}(M,\nabla)$ and that the parallel flow of $X$ is equal to its flow. For an arbitrary $m\in M$, we have $\varepsilon\in\R^+$ such that $\Psi^X_{2t}$ is defined for all $t\in(-\varepsilon, \varepsilon)$, hence, the equality of the flows yields
 \begin{equation}\label{eq: aff0-prop}
 \rest{\frac{\dif}{\dif t}}{t=0}\Theta^X_{t,m}Y_{\Psi^X_t(m)}=\rest{\frac{\dif}{\dif t}}{t=0}(\Psi^X_{-t})_{*\Psi^X_{t}(m)}Y_{\Psi^X_t(m)}
 \end{equation}
 for every $Y\in\mathfrak{X}(M)$. By \eqref{eq: sym-br-par-flow} and the fact that the right-hand side is equal to $\rest{[X,Y]}{m}$, we get that $X$ is indeed parallel:
 \begin{equation*}
 0=\rest{[X,Y]_s}{m}-\rest{[X,Y]}{m}=2\nabla_{Y_m}X.
 \end{equation*}

 Let us now assume that $X\in\mathrm{aff}_0(M,\nabla)$. As $\mathrm{aff}_0(M,\nabla)$ is a subalgebra of $\mathrm{aff}(M,\nabla)$, it is enough to prove that the parallel flow of $X$ is equal to its flow. For an arbitrary point $m\in M$, we take a basis $\{e_i\}_{i=1}^n$ for $T_mM$ and extend it by the parallel transport to every point of the integral curve $\gamma(t)\coloneqq\Psi^X_{t}(m)$. In particular, we have the basis $\{e_i(t)\}_{i=1}^n$ for $T_{\gamma(t)}M$ for every $t\in\R$ such that $\Psi^X_{2t}(m)$ is defined. For $i\in\{1,\dots, n\}$, we can now introduce the curve in $T_mM$:
 \begin{equation*}
 Z_i(t)\coloneqq (\Theta^X_{t,m}-(\Psi^X_{-t})_{*\gamma(t)})e_i(t)
 \end{equation*}
 defined for every $t\in\R$ such that $\Psi^X_{2t}(m)$ is defined. In order to show that the flows are equal, it is enough to prove that $\{Z_i\}_{i=1}^n$ are constant curves equal to $0$. Clearly, we have that $Z_i(0)=0$ and, moreover, using properties of the parallel transport and the flow of a vector field yields
 \begin{align*}
 \dot{Z}_i(t)=\,&\rest{\frac{\dif}{\dif \lambda}}{\lambda=0}Z_i(t+\lambda)=\rest{\frac{\dif}{\dif \lambda}}{\lambda=0}(\Theta^X_{t+\lambda,m}-(\Psi^X_{-(t+\lambda)})_{*\gamma(t+\lambda)})e_i(t+\lambda)\\
 =\,&\rest{\frac{\dif}{\dif\lambda}}{\lambda=0}P^\gamma_{2t,0}P^\gamma_{2(t+\lambda),2t}(\Psi^X_t)_{*\gamma(t+2\lambda)}(\Psi^X_\lambda)_{*\gamma(t+\lambda)}e_i(t+\lambda)\\
 &-\rest{\frac{\dif}{\dif\lambda}}{\lambda=0}(\Psi^X_{-t})_{*\gamma(t)}(\Psi^X_{-\lambda})_{*\gamma(t+\lambda)}e_i(t+\lambda).
 \end{align*}
 By definition, the flow of $X\in\mathrm{aff}(M,\nabla)$ is a local $1$-parameter subgroup of $\aff(M,\nabla)$, that is, the tangent map of the flow of $X$ commutes with the parallel transport (Definition \ref{def: affine}), hence $\dot{Z}_i(t)$ is equal to
 \begin{align*}
 P^\gamma_{2t,0}(\Psi^X_t)_{*\gamma(t)}\rest{\frac{\dif}{\dif\lambda}}{\lambda=0}P^{\gamma^t}_{2\lambda,0}(\Psi^X_\lambda)_{*\gamma^t(\lambda)}e^t_i(\lambda)-(\Psi^X_{-t})_{*\gamma(t)}\rest{\frac{\dif}{\dif\lambda}}{\lambda=0}(\Psi^X_{-\lambda})_{*\gamma^t(\lambda)}e^t_i(\lambda),
 \end{align*}
where $\gamma^t(\lambda)\coloneqq\gamma(t+\lambda)$ and $e^t_i(\lambda)\coloneqq e_i(t+\lambda)$. By the torsion-freeness and the fact that $X$ is parallel, we have that
\begin{equation*}
 \dot{Z}_i(t)=P^\gamma_{2t,0}(\Psi^X_t)_{*\gamma(t)}\rest{[X,e_i]_s}{\gamma(t)}-(\Psi^X_{-t})_{*\gamma(t)}\rest{[X,e_i]}{\gamma(t)}=(\Theta^X_{t,m}-(\Psi^X_{-t})_{*\gamma(t)})\nabla_{\dot{\gamma}(t)} e_i.
\end{equation*}
By the construction $\{ e_i\}_{i=1}^n$ are parallel along $\gamma$, and thus we obtain that $\dot{Z}_i(t)=0$. In conclusion, we have that $Z_i(t)=0$ and the result follows.
\end{proof}

We finish this section by one of the main result of this paper. It gives the necessary and sufficient condition under which commutation relations of symmetric Cartan calculus are simplified to a natural shape. 

\begin{theorem}\label{thm: commutation-relations-sym-Cartan}
 Let $\nabla$ be a torsion-free connection on $M$ and $X,Y\in\mathfrak{X}(M)$. All of the six relations
 \begin{align*}
 [\iota_X,\iota_Y]&=0,& [\nabla^s,\nabla^s]&=0,&
 [\iota_X,\nabla^s]&=L^s_X,\\
 [L^s_X,\iota_Y]&=\iota_{[ X,Y]_s}, & [L^s_X,L^s_Y]&=L^s_{[X,Y]_s}, & [\nabla^s, L^s_X]&=0
\end{align*}
are satisfied if and only if $X\in\aff_0(M,\nabla)$.
\end{theorem}

\begin{proof}
 Note that the four commutators
 \begin{align*}
 [\iota_X,\iota_Y]&=0,& [\nabla^s,\nabla^s]&=0,&
 [\iota_X,\nabla^s]&=L^s_X, &
 [L^s_X,\iota_Y]&=\iota_{[ X,Y]_s}.
\end{align*}
are true for any vector fields. By Proposition \ref{prop:commutator-sym-der-sym-Lie}, we have that, in general,
\begin{equation*}
[\nabla^s,L^s_X]=2\nabla^s_{\nabla X}+\iota^s_{2\sym \iota_XR_\nabla +R^s_\nabla X}.
\end{equation*}
We clearly have that $\rest{[\nabla^s,L^s_X]}{\cCi(M)}=0$ if and only if $\rest{(\nabla^s_{\nabla X})}{\cCi(M)}=0$, which is equivalent to $X$ being parallel. For a parallel vector field $X$, we have that $ \rest{[\nabla^s,L^s_X]}{\Upsilon^1(M)}=0$ if and only if $2\sym \iota_XR_\nabla +R^s_\nabla X=0$, that is, by Proposition \ref{prop: affine vector field tf}, $X$ is an affine vector field. Altogether,
\begin{equation*}
 [\nabla^s,L^s_X]=0
\end{equation*}
if and only if $X\in\aff_0(M,\nabla)$. 

By Proposition \ref{prop: com1}, we have that, in general,
\begin{equation*}
 [L^s_X, L^s_Y]=L^s_{[X,Y]}+\iota^s_{2(\nabla_{\nabla X}Y-\nabla_{\nabla Y}X-R_\nabla (X,Y))}.
\end{equation*}
If $X$ is parallel, we have $[X,Y]=[X,Y]_s$ and the identity
\begin{equation*}
 [L^s_X, L^s_Y]=L^s_{[X,Y]_s}
\end{equation*}
is satisfied if and only if $R_\nabla (X,Y)=0$. Moreover, $X\in\aff_0(M,\nabla)$ is clearly equivalent to $R^s_\nabla X=0$ and $\sym\iota_X R_\nabla =0$. Therefore and by and the algebraic Bianchi identity, we have, for all $Y,Z\in\mathfrak{X}(M)$, that
\begin{equation*}
 R_\nabla (X,Y)Z=\frac{1}{2}(R_\nabla (X,Y)Z-R_\nabla (X,Z)Y)=-\frac{1}{2}R_\nabla (Y,Z)X,
\end{equation*}
which vanishes because $X$ is parallel and the result follows.
\end{proof}

\section{The Patterson-Walker metric and Killing vector fields}\label{sec:Patterson-Walker}


The commutation relations of symmetric Cartan calculus in Theorem \ref{thm: commutation-relations-sym-Cartan} are completely analogous to those of classical Cartan calculus, but they require an extra condition. At first glance, this may seem as an asymmetry between the two theories. However, this apparent discrepancy can be fully explained via the so-called Patterson-Walker metric (originally called the Riemannian extension in \cite{PatRE}). This is a metric that was originally introduced for the study of metrics admitting parallel fields of partially null planes and has had recent applications in the construction of anti-self-dual Einstein metrics in dimension four \cite{dunajski-mettler}.

We first reintroduce this metric in a coordinate-free way, then express it in terms of symmetric Cartan calculus and finally spell out the link to Theorem \ref{thm: commutation-relations-sym-Cartan} using Killing vector fields for the Patterson-Walker metric. Although the statements of the main results are neat, the proofs require some heavy lifting.

\subsection{Definition of the Patterson-Walker metric}
For an arbitrary connection $\nabla$ on $M$, the induced connection on the cotangent bundle $\pr\colon T^*M\rightarrow M$ gives the decomposition 
\begin{equation*}
T(T^*M)=\mathcal{H}\nb\oplus \mathcal{V}\cong \pr^*TM\oplus \pr^*T^*M = \pr^*(TM\oplus T^*M)
\end{equation*}
and the $\cCi(M)$-module morphism $\phi\nb\colon \Gamma(TM\oplus T^*M)\rightarrow \mathfrak{X}(T^*M)$.

\begin{remark}
 The structure of $\cCi(M)$-module on $\mathfrak{X}(T^*M)$ is given, for $f\in \cCi(M)$ and $Q\in\mathfrak{X}(T^*M)$, by $f\cdot Q\coloneqq (\pr^*f)\,Q$. The morphism $\phi\nb$ is given by pullback: for $\zeta\in T^*M$ and $a\in\Gamma(TM\oplus T^*M)$,
\begin{equation*}
 (\phi\nb a)(\zeta)\coloneqq a(\pr(\zeta)).
\end{equation*}
\end{remark}

Explicitly, $\rest{\phi\nb}{\mathfrak{X}(M)}$ is the horizontal lift, which we will denote $X^\text{h}\coloneqq \phi\nb X$, and $\rest{\phi\nb}{\Omega^1(M)}$ is the vertical lift, which we will denote $\alpha^\text{v}\coloneqq \phi\nb\alpha$. More details about lifts to the cotangent bundle are found in Appendix \ref{app: lifts}.

The bundle $TM\oplus T^*M$ comes equipped with a canonical symmetric pairing
\begin{equation}\label{eq:pairing}
 \la X+\alpha, Y+\beta\ra_+\coloneqq \alpha(Y)+\beta(X).
\end{equation}
Since pullback sections locally generate the entire space of sections, the image of $\phi\nb$ generates $\mathfrak{X}(T^*M)$ locally. Therefore, we can bring $\la\,\,,\,\,\ra_+$ to $T^*M$:

\begin{definition} \label{def:PW-metric} Let $\nabla$ be a connection on $M$. \textbf{The Patterson-Walker metric} $g\nb$ is the split-signature metric on $T^*M$ determined by
 \begin{equation*}
 g\nb(\phi\nb a,\phi\nb b)\coloneqq \pr^*\la a,b\ra_+.
\end{equation*}
\end{definition}

Different connections can induce different Patterson-Walker metrics. Let us show when the Patterson-Walker metrics of two different connections coincide. First, we need to prove two lemmas.

\begin{lemma}\label{lem: hor lifts}
 Given two connections on $M$ that differ by $\tau\in\Gamma(\otimes^2T^*M\otimes TM)$, their corresponding horizontal lifts $(\,\,)^\emph{h}$ and $(\,\,)^{\emph{h}'}$ are related as follows:
 \begin{equation*}
 X^{\emph{h}'}=X^\emph{h}+(\iota_X\tau)^\upsilon,
 \end{equation*}
 where $\upsilon$ denotes the vertical lift of sections of $\en TM$ (see Appendix \ref{app: lifts}).
\end{lemma}

\begin{proof}
 By Lemma \ref{lem: hor and ver}, we have 
 \begin{equation*}
 X^{\text{h}'}Y^v=(\nabla'_XY)^v=(\nabla_XY+\tau(X,Y))^v=(X^\text{h}+(\iota_X\tau)^\upsilon)Y^v
 \end{equation*}
 for any $Y\in\mathfrak{X}(M)$. Therefore, by Proposition \ref{prop: vliftvfield}, $X^{\text{h}'}=X^\text{h}+(\iota_X\tau)^\upsilon$.
\end{proof}

\begin{lemma}\label{lem: PW and End}
 Let $\nabla$ be a connection on $M$. For $A\in\Gamma(\en TM)$, we have
 \begin{align*}
 g\nb(A^\upsilon ,\alpha^\emph{v})&=0, & g\nb(A^\upsilon ,X^\text{h})=(AX)^v.
 \end{align*}
\end{lemma}

\begin{proof}
 The first identity follows easily from the fact that vertical subbundle is isotropic. For the second one, in natural coordinates $(T^*U, \lbrace x^i\rbrace\cup\lbrace p_j\rbrace)$, we have
 \begin{align*}
 \rest{A}{U}&=A^i_j\dif x^j\otimes\partial_{x^i}, & \rest{A^\upsilon }{T^*U}&=p_iA^i_j\partial_{p_j}=p_i(\iota_{\dif x^i}A)^\text{v}.
 \end{align*}
 Therefore, $\rest{g\nb(A^\upsilon , X^\text{h})}{T^*U}=p_i\dif x^i(AX)=\rest{(AX)^v}{T^*U}$.
\end{proof}

\begin{proposition}\label{prop: PW and PW'}
 Given two connections on $M$, their Patterson-Walker metrics coincide if and only if the associated torsion-free connections are the same.
\end{proposition}

\begin{proof}
 Let $\nabla$ and $\nabla'$ be two connections on $M$ related to each other by the difference tensor field $\tau\in\Gamma(\otimes^2T^*M\otimes TM)$. Clearly, $\rest{g\nb}{\Sym^2\mathcal{V}}=0=\rest{g_{\s{$\nabla'$}}}{\Sym^2\mathcal{V}}$. Thanks to Lemma \ref{lem: hor lifts} and that $\mathcal{V}$ is isotropic, we find
 \begin{equation*}
 g\nb(X^{\text{h}'},\beta^v)=g\nb(X^\text{h}+(\iota_X\tau)^\upsilon,\beta^\text{v})=\pr^*\beta(X)=g_{\s{$\nabla'$}}(X^{\text{h}'},\beta^v).
 \end{equation*}
 Taking into account the fact that also the horizontal subbundle is isotropic,
 \begin{equation*}
 g\nb(X^{\text{h}'},Y^{\text{h}'})=g\nb(X^\text{h}+(\iota_X\tau)^\upsilon,Y^\text{h}+(\iota_Y\tau)^\upsilon)=g\nb(X^{\text{h}},(\iota_Y\tau)^\upsilon)+g\nb((\iota_X\tau)^\upsilon,Y^\text{h}).
 \end{equation*}
 It follows from Lemma \ref{lem: PW and End} that
 \begin{equation*}
 g\nb(X^{\text{h}'},Y^{\text{h}'})=(\tau(Y,X)+\tau(X,Y))^v,
 \end{equation*}
 which vanishes (i.e., it is equal to $g_{\nabla'}(X^{\text{h}'}, Y^{\text{h}'})$) if and only if $\tau\in\Omega^2(M, TM)$. Equivalently, $\nabla'^0=\nabla^0$ because
 \begin{equation*}
 \nabla'^0_XY-\nabla^0_XY=\frac{1}{2}(\pg{X,Y}'_s-[X,Y]_s)=\frac{1}{2}(\tau(X,Y)+\tau(Y,X)).
 \end{equation*}
\end{proof}

Thanks to Proposition \ref{prop: PW and PW'}, we can, without loss of generality, consider only torsion-free connections, as we do for symmetric Cartan calculus.

\begin{remark} In natural local coordinates $(T^*U,\lbrace x^j\rbrace\cup\lbrace p_j\rbrace)$ on $T^*M$, Definition \ref{def:PW-metric} gives that the Patterson-Walker metric takes the form
\begin{equation}\label{eq:PW-coordinates}
 \rest{g\nb}{T^*U}=\dif p_i\odot\dif x^i-p_k(\pr^*\Gamma^k_{ij})\,\dif x^i\odot\dif x^j,
\end{equation}
where $\lbrace\Gamma^k_{ij}\rbrace$ are the Christoffel symbols of $\nabla$ in the chart $(U,\lbrace x^j\rbrace)$. 
This recovers the original definition in \cite{PatRE}. 
\end{remark}

\subsection{Reformulation using symmetric Cartan calculus}\label{sec:reformulation-sym-Cartan-calculus} We use a natural lift of the connection $\nabla$ from $M$: the connection $\hat{\nabla}$ on the manifold $T^*M$ that is uniquely determined by

\begin{equation}\label{eq: lift of nabla}
 \begin{aligned}
 \hat{\nabla}_{X^\text{h}}Y^\text{h}&\coloneqq (\nabla_XY)^\text{h}, & &&&&&&\hat{\nabla}_{\alpha^\text{v}}X^\text{h}&\coloneqq 0,\\
 \hat{\nabla}_{X^\text{h}}\alpha^\text{v}&\coloneqq (\nabla_X\alpha)^\text{v}, & &&&&&&\hat{\nabla}_{\alpha^\text{v}}\beta^\text{v}&\coloneqq 0.
\end{aligned}
\end{equation}

 With the lift $\hat{\nabla}$ we can prove the following.
 
\begin{theorem}\label{thm: PW}
 For any connection $\nabla$ on $M$, there holds
 \begin{equation}\label{eq:nablas-acan-gnabla}
\hat{\nabla}^s\alpha_\emph{can}=g\nb,
 \end{equation}
 where $\alpha_\text{can}$ is the canonical $1$-form on $T^*M$.
\end{theorem}

\begin{proof}
 We use the properties of the lifts that are proved in Appendix \ref{app: lifts}. By Lemmas \ref{lem: hor and ver coord} and \ref{lem: hor and ver}, we find
 \begin{align*}
X^\text{h}\alpha_\text{can}(Y^\text{h})&=X^\text{h}Y^v=(\nabla_XY)^v=\alpha_\text{can}((\nabla_XY)^\text{h}),\\
\beta^\text{v}\alpha_\text{can}(X^\text{h})&=\beta^\text{v}X^v=\pr^*\beta(X).
 \end{align*}
Using these together with the fact that $\alpha_\text{can}$ annihilates $\mathcal{V}$, we get
 \begin{align*}
 (\hat{\nabla}^s\alpha_\text{can})(X^\text{h}, Y^\text{h})&=\alpha_\text{can}(\pg{X,Y}^\text{h}_s-\pg{X^\text{h},Y^\text{h}}_{\hat{\nabla}^s}),\\
 (\hat{\nabla}^s\alpha_\text{can})(\alpha^\text{v},\beta^\text{v})&=-\alpha_\text{can}(\pg{\alpha^\text{v},\beta^\text{v}}_{\hat{\nabla}^s}),\\
 (\hat{\nabla}^s\alpha_\text{can})(X^\text{h},\beta^\text{v})&=\pr^*\beta(X)-\alpha_\text{can}(\pg{X^\text{h},\beta^\text{v}}_{\hat{\nabla}^s}),
 \end{align*}
 where $\pg{\,\,,\,\,}_{\hat{\nabla}^s}$ stands for the symmetric bracket corresponding to $\hat{\nabla}$. The terms with $\alpha_\text{can}$ on the right-hand sides of the above equations vanish by the definition of the connection $\hat{\nabla}$ and the result follows.
 \end{proof}

 Note that the connection $\hat{\nabla}$ defined by \eqref{eq: lift of nabla} is not the unique connection on $T^*M$ satisfying \eqref{eq:nablas-acan-gnabla}. To start with, we can take any other connection on $T^*M$ with the same associated torsion-free connection. However, there is actually much more freedom.

 \begin{lemma}\label{lem:con-D-satisfying-nablas-acan-gnabla}
 A connection $\mathcal{D}$ on $T^*M$ satisfies \eqref{eq:nablas-acan-gnabla} if and only if
\begin{equation}\label{eq: nablas PW condition}\pg{X^\emph{h}+\alpha^\emph{v},Y^\emph{h}+\beta^\emph{v}}_{\mathcal{D}^s}-[X,Y]_s^\emph{h}\in\Gamma(\mathcal{V}),
\end{equation}
where $\pg{\,\,,\,\,}_{\mathcal{D}^s}$ stands for the symmetric bracket corresponding to $\mathcal{D}$.
 \end{lemma}

\begin{proof}
 If we follow the same approach as in the proof of Theorem \ref{thm: PW}, we find that the claim is true because the kernel of $\alpha_\text{can}$ is precisely the vertical subbundle.
\end{proof}

 As both $g\nb$ and $\hat{\nabla}$ are naturally associated to the connection $\nabla$ on $M$, it is reasonable to wonder what the relation between $\hat{\nabla}$ and the Levi-Civita connection of $g\nb$ is. We first compute the torsion of the connection $\hat{\nabla}$.

\begin{lemma}\label{lem: hat nabla is not LC}
 Let $\nabla$ be a connection on $M$. For the corresponding connection $\hat{\nabla}$ on $T^*M$, we have that $\hat{\nabla}g\nb=0$, $\rest{T_{\hat{\nabla}}}{\mathcal{H}\nb\wedge \mathcal{V}}=0$, $\rest{T_{\hat{\nabla}}}{\mathcal{V}\wedge \mathcal{V}}=0$ and
 \begin{equation*}
 T_{\hat{\nabla}}(X^\emph{h},Y^\emph{h})=T_\nabla (X,Y)^\emph{h}-R_\nabla (X,Y)^\upsilon.
\end{equation*}
\end{lemma}

\begin{proof}
 The claim for the torsion follows directly from Lemma \ref{lem: Lie brackets hor and ver}. 
 
 For $\hat{\nabla}g\nb=0$, we use the shorthand notation $a,b,c\in\Gamma(TM\oplus T^*M)$ and 
 \begin{equation}\label{eq: nabla gen}
 \nabla^\text{gen}_{X+\alpha}(Y+\beta)\coloneqq \nabla_XY+\nabla_X\beta.
 \end{equation}
 The superscript `\text{gen}' will be explained in Section \ref{sec:relation-gen-geo}. We also use $\pr_1$ and $\pr_2$ for the projections to $TM$ and $T^*M$, respectively. 
 
 With this notation, we have the following for $\nabla^\text{gen}$:
 \begin{equation*}
 (\pr_1 a)\la b,c\ra_+-\la\nabla^\text{gen}_a\,b,c\ra_+-\la b,\nabla^\text{gen}_a\,c\ra_+=0.
 \end{equation*}
 Therefore, it follows from \eqref{eq: nabla ca 2} that
 \begin{align*}
 (\hat{\nabla}g\nb)&(\phi\nb a,\phi\nb b,\phi\nb c)\\
 &=(\phi\nb a)g\nb(\phi\nb b, \phi\nb c)-g\nb(\hat{\nabla}_{\phi\nb a}\phi\nb b,\phi\nb c)-g\nb(\phi\nb b,\hat{\nabla}_{\phi\nb a}\phi\nb c)\\
 &=((\pr_1 a)^\text{h}+(\pr_2 a)^\text{v})\pr^*\la b,c\ra_+-g\nb(\phi\nb\nabla^\text{gen}_a\,b,\phi\nb c)-g\nb(\phi\nb b,\phi\nb\nabla^\text{gen}_a\,c),
 \end{align*}
 By Lemma \ref{lem: hor and ver coord},
 \begin{equation*}
 (\hat{\nabla}g\nb)(\phi\nb a,\phi\nb b,\phi\nb c)=\pr^*((\pr_1 a)\la b,c\ra_+-\la \nabla^\text{gen}_a\,b,c\ra_+-\la b,\nabla^\text{gen}_a\,c\ra_+)=0.
 \end{equation*}
\end{proof}

We can now slightly modify $\hat{\nabla}$ to recover the Levi-Civita connection of $g\nb$. Recall that we can restrict to the torsion-free case without loss of generality.

 \begin{proposition}\label{prop: PWLC}
 Let $\nabla$ be a torsion-free connection on $M$. Then the connection $\bar{\nabla}$ on $T^*M$ given by
 \begin{equation*}
\bar{\nabla}_{\phi\nb a}\phi\nb b\coloneqq \hat{\nabla}_{\phi\nb a}\phi\nb b-(R_\nabla (\pr_1 b,\,\,)(\pr_1 a))^\upsilon.
 \end{equation*}
 is the Levi-Civita connection of the Patterson-Walker metric $g\nb$.
 \end{proposition}

 \begin{proof}
 It follows directly from Lemma \ref{lem: hat nabla is not LC} that $\rest{T_{\bar{\nabla}}}{\mathcal{H}\nb\wedge \mathcal{V}}=0$, $\rest{T_{\bar{\nabla}}}{\mathcal{V}\wedge \mathcal{V}}=0$ and
 \begin{equation*}
 T_{\bar{\nabla}}(X^\text{h},Y^\text{h})=T_\nabla (X,Y)^\text{h}-R_\nabla (X,Y)^\upsilon-(R_\nabla (Y,\,\,)X)^\upsilon+(R_\nabla (X,\,\,)Y)^\upsilon,
 \end{equation*}
 which vanishes by the torsion-freeness of $\nabla$ and the algebraic Bianchi identity. Using $\hat{\nabla}g\nb=0$, see Lemma \ref{lem: hat nabla is not LC}, and that $\mathcal{V}$ is isotropic, we find
 \begin{align*}
 (\bar{\nabla}g\nb)&(\phi\nb a,\phi\nb b,\phi\nb c)\\
 &=g\nb((R_\nabla (\pr_1 b,\,\,)(\pr_1 a))^\upsilon,(\pr_1 c)^\text{h})+g\nb((\pr_1 b)^\text{h},(R_\nabla (\pr_1 c,\,\,)(\pr_1 a))^\upsilon).
 \end{align*}
 Finally, using Lemma \ref{lem: PW and End}, we find
 \begin{equation*}
(\bar{\nabla}g\nb)(\phi\nb a,\phi\nb b,\phi\nb c)=(R_\nabla (\pr_1 b,\pr_1 c)\pr_1 a+R_\nabla (\pr_1 c,\pr_1 b)\pr_1a)^v=0.
 \end{equation*}
 \end{proof}

A subtle point here is that the connection $\bar{\nabla}=\lcn{g\nb}$ is not the associated torsion-free connection to $\hat{\nabla}$. But it clearly satisfies condition \eqref{eq: nablas PW condition} in Lemma \ref{lem:con-D-satisfying-nablas-acan-gnabla}, so
\begin{equation*}
\bar{\nabla}^s\alpha_\text{can}=g\nb.
\end{equation*}
 
\begin{remark}\label{rem:PW-bracket}
Another reformulation of the Patterson-Walker metric using the symmetric bracket can be directly deduced from \cite[\S 17]{Patlift}, but it involves the notions of vertical lift $X^v\in \cCi(T^*M)$ and complete lift $X^c\in\mathfrak{X}(T^*M)$ of a vector field $X\in\mathfrak{X}(M)$. They can be defined by $X^v(\zeta)\coloneqq \zeta(X_{\pr(\zeta)})$
for $\zeta\in T^*M$ and $X^c=-\ham X^v$ (see Appendix \ref{app: lifts} for more details). With these notions, one has that $g\nb$ is determined (thanks to Proposition \ref{prop-determined-by-complete}) by
 \begin{equation}\label{eq: PW2}
 g\nb(X^c,Y^c)=-[X,Y]_s^v.
 \end{equation}
\end{remark}

 Note the striking similarity between the Patterson-Walker metric and the canonical symplectic form $\omega_\text{can}$:
 \begin{align*}
 \rest{\omega_\text{can}}{T^*U}&=\dif p_i\wedge\dif x^i, & \dif\alpha_\text{can}&=\omega_\text{can}, & \omega_\text{can}(X^c,Y^c)&=-[X,Y]^v. 
 \end{align*}
 This analogy is explained by the fact that $\omega_\text{can}$ may be alternatively defined by 
 \begin{equation*}
 \omega_\text{can}(\phi\nb a,\phi\nb b)=\pr^*\la a,b\ra_-,
 \end{equation*}
 where $\la \,\,,\,\,\ra_-$ is the canonical skew-symmetric pairing on $TM\oplus T^*M$:
 \begin{equation}\label{eq:skew-sym-pairing}
 \la X+\alpha,Y+\beta\ra_-\coloneqq \alpha(Y)-\beta(X). 
 \end{equation}
 Unlike the Patterson-Walker metric, the canonical symplectic form is independent of the choice of torsion-free connection on $M$. However, considering connections with non-vanishing torsion gives rise to different non-degenerate $2$-forms on $T^*M$, which are closed if and only if the connection is torsion-free.

\subsection{Relation to generalized geometry}\label{sec:relation-gen-geo}

 We briefly hint the relation to generalized geometry (in the sense of \cite{HitGCYM, GuaGCG}), which will be explored in forthcoming works. This section can be skipped on a first reading.

 To start with, the bundle $TM\oplus T^*M$ and the canonical symmetric pairing \eqref{eq:pairing} is the starting point of generalized geometry, sometimes referred to as the generalized tangent bundle. The skew-symmetric pairing \eqref{eq:skew-sym-pairing} is equally canonical, although its role has remained largely unexplored (we will actually explore it in future work).

Using the language of generalized geometry, we can justify the naturality of the connection $\hat{\nabla}$ defined by \eqref{eq: lift of nabla}. The key concept is that of a generalized connection (see, e.g., \cite{gualtieri-branes}, although this concept goes back to the unpublished notes \cite{AleDBCA}). For $E\coloneqq TM\oplus T^*M$, a generalized connection is a linear operator \[D\colon \Gamma(E)\to \Gamma(E^*\otimes E)\] such that, with the notation $a,b\in \Gamma(E)$ and $f\in \cCi(M)$, 
\begin{align*}
 D(fa)&=fDa+df\otimes a,& d\la a,b\ra_+=\la Da,b\ra_+ + \la a, Db\ra_+.
\end{align*}
For any connection $\nabla$ on $M$, the connection $\nabla^\text{gen}$ defined by \eqref{eq: nabla gen} in the proof of Lemma \ref{lem: hat nabla is not LC} is a natural generalized connection (and hence the superscript `gen') associated to the connection $\nabla$. 

With this language, $\hat{\nabla}$ is then fully determined by the natural formula
 \begin{equation}\label{eq: nabla ca 2}
 \hat{\nabla}_{\phi\nb a}\phi\nb b=\phi\nb(\nabla^\text{gen}_a\,b).
 \end{equation}
 
 \begin{remark}
 Note that this construction can be easily generalized by replacing $\nabla^\text{gen}$ with an arbitrary generalized connection, which would lead, through Theorem \ref{thm: PW}, to a generalization of the Patterson-Walker metric. 
 \end{remark}

\subsection{Killing vector fields for the Patterson-Walker metric}
Killing vector fields for the Patterson-Walker metric can be characterized in terms of lifts of tensor fields on the base manifold $M$. This result was first established in \cite[Prop. 6.8]{conformal-PW}. We restate it in a coordinate-free form and provide an alternative, more direct proof using the formalism we have developed and an intrinsic approach to lifts (see Appendix \ref{app: lifts}).

\begin{theorem}\label{thm: Kill-PW}
 Let $\nabla$ be a torsion-free connection on $M$. The vector space of Killing vector fields for $g\nb$ admits the natural decomposition:
 \begin{equation*}
 \{\alpha^\emph{v}\,|\,\alpha\in\kil^1(M,\nabla)\}\oplus \{X^c\,|\,X\in\aff(M,\nabla)\}\oplus\left\{\begin{array}{c|c}
 \!\! \pi^h & \begin{array}{c}
 \pi\in\mathfrak{X}^2(M) \text{ such that } \\
 \nabla\pi=0 \text{ and }R^\pi_\nabla=0
 \end{array}
 \end{array}\!\!\!\!\!\right\},
 \end{equation*}
 where $R^\pi_\nabla\in\Upsilon^2(M,\Sym^2TM)$ is given by
 \begin{equation*}
 R^\pi_\nabla(Y,Z)\coloneqq\sym(R\nb(Y,\pi(\,))Z+R\nb(Z,\pi(\,))Y).
 \end{equation*}
\end{theorem}

\begin{proof}
 A vector field $Q\in\mathfrak{X}(T^*M)$ is Killing for $g\nb$ if and only if 
 \begin{equation}\label{eq: Q-Kill}
 0=\bar{\nabla}^sg\nb(Q)=:C\in\Upsilon^2(T^*M). 
 \end{equation}
 In natural coordinates, the vector field $Q$ is fully determined by $2n$-tuple of functions $\lbrace Q^i\rbrace\cup\lbrace Q_i\rbrace\subseteq \cCi(T^*U)$ as 
 \begin{equation*}
 \rest{Q}{T^*U}=Q^i\partial_{x^i}^\text{h}+Q_i\partial_{p_i},
 \end{equation*}
 and the condition \eqref{eq: Q-Kill} becomes the following system of $\frac{3}{2}n(n+1)$ PDEs:
\begin{align}
 C(\partial_{p_i},\partial_{p_j})&=\frac{\partial Q^i}{\partial p_j}+\frac{\partial Q^j}{\partial p_i}=0,\label{eq: PW-1}\\
 C(\partial^\text{h}_{x^i},\partial_{p_j})&=\partial^\text{h}_{x^i}Q^j+\frac{\partial Q_i}{\partial p_j}+Q^k\pr^*\Gamma^j_{ki}=0,\label{eq: PW-2}\\
 C(\partial^\text{h}_{x^i},\partial^\text{h}_{x^j})&=\partial^\text{h}_{x^j}Q_i+\partial^\text{h}_{x^i}Q_j-2Q_k\pr^*\Gamma^k_{ji}+p_lQ^k\pr^*((R_\nabla )^l_{ijk}+(R_\nabla )^l_{jik})=0,\label{eq: PW-3}
\end{align}
where $(R_\nabla )^l_{ijk}\coloneqq \dif x^l(R_\nabla (\partial_{x^j},\partial_{x^k})\partial_{x^i})$. Let us now find the general solution to this system. Differentiating \eqref{eq: PW-1} with respect to the variable $p_k$ yields
\begin{equation*}
 0=\frac{\partial^2 Q^i}{\partial p_k\partial p_j}+\frac{\partial^2 Q^j}{\partial p_k\partial p_i}=\frac{\partial^2 Q^i}{\partial p_j\partial p_k}+\frac{\partial^2 Q^j}{\partial p_i\partial p_k}.
\end{equation*}
Using \eqref{eq: PW-1} on the above equation and swapping the partial derivatives, we get
\begin{equation*}
 0=-2\frac{\partial^2 Q^k}{\partial p_i\partial p_j},
\end{equation*}
which has the general solution $Q^i=p_k\,\pr^*\pi^{ki}+\pr^*X^i$ for some functions $\lbrace \pi^{ki}\rbrace\cup\lbrace X^i\rbrace\subseteq \cCi(U)$. It is easy to see that the general solution $\lbrace Q^i\rbrace$ of \eqref{eq: PW-1} is of the aforementioned form with the extra condition $\pi^{ki}=-\pi^{ik}$. It follows from the transformation law for the components of $Q\in\mathfrak{X}(T^*M)$,
\begin{equation*}
 \tilde{Q}^i=\frac{\partial \tilde{x}^i}{\partial x^k}Q^k,
\end{equation*}
that the solutions to \eqref{eq: PW-1} described by $\lbrace \pi^{kj}\rbrace\cup\lbrace X^k\rbrace$ patch together to a global solution if and only if $\lbrace\pi^{kj}\rbrace$ are components of a global bivector field $\pi\in\mathfrak{X}^2(M)$ and $\lbrace X^k\rbrace$ are components of a global vector field $X\in\mathfrak{X}(M)$. Comparing with the local coordinate expressions for the horizontal lifts of a bivector field and a vector field (Remarks \ref{rem:coord-hor-bivector} and \ref{rem:coord-hor-vector}), we conclude that
\begin{equation*}
 \pr_{\mathcal{H}\nb}Q=\pi^h+X^\text{h}.
 \end{equation*}

By plugging in the general solution of \eqref{eq: PW-1} into \eqref{eq: PW-2}, we obtain
\begin{align*}
 \frac{\partial Q_i}{\partial p_j}&=-\pr^*\Big(\frac{\partial X^{j}}{\partial x^i}+X^k\Gamma^j_{ki}\Big)-p_k\,\pr^*\Big(\frac{\partial\pi^{jk}}{\partial x^i}+\Gamma^k_{li}\pi^{jl}+\Gamma^j_{li}\pi^{lk}\Big)\\
 &=-\pr^*(\nabla X)^j_i-p_k\,\pr^*(\nabla\pi)^{jk}_i.
\end{align*}
Taking the partial derivative with respect to the variable $p_k$ and subsequently $p_l$, we arrive to
\begin{equation*}
 \frac{\partial^3 Q_i}{\partial_{p_l}\partial_{p_k}\partial_{p_j}}=0,
\end{equation*}
the general solution of which is parametrized by $\lbrace A^{kl}_i\rbrace\cup\lbrace B^l_i\rbrace\cup\lbrace\alpha_i\rbrace\subseteq\cCi(U)$ as $Q_i=p_kp_l\,\pr^*A^{kl}_i+p_l\,\pr^*B^{l}_i+\pr^*\alpha_j$. Plugging this back to \eqref{eq: PW-2} yields
\begin{align*}
 B^j_i&=-(\nabla X)^j_i, & p_k\,\pr^*(A^{kj}_i+A^{jk}_i)&=-p_k\,\pr^*(\nabla\pi)^{jk}_i.
\end{align*}
\sloppy As the left-hand side of the second equation is symmetric on $(j,k)$ but the \mbox{right-hand} side is skew-symmetric on $(j,k)$, both sides must vanish. The general solution to \eqref{eq: PW-1} and \eqref{eq: PW-2} is thus given by $\pi\in\mathfrak{X}^2(M)$, $X\in\mathfrak{X}(M)$ and $\lbrace\alpha_i\rbrace$~as
\begin{align*}
 \rest{Q}{T^*U}&=\rest{(\pi^h+X^\text{h})}{T^*U}+(-p_j\,\pr^*(\nabla X)^{j}_i+\pr^*\alpha_i)\partial_{p_i}
\end{align*}
together with extra condition $\nabla\pi=0$. Using the transformation law
\begin{equation*}
 \tilde{Q}_j=\frac{\partial x^k}{\partial \tilde{x}^j}Q_k,
\end{equation*}
we find that the local solutions patch together to a global one if and only if the functions $\lbrace\alpha_i\rbrace$ are the components of a $1$-form $\alpha\in\Omega^1(M)$. Using the local coordinate expressions for the vertical lifts of a field of endomorphisms and a $1$-form (Remark \ref{rem:coord-hor-vector}), we have the global formula
\begin{equation*}
 Q=\pi^h+X^\text{h}-(\nabla X)^\upsilon+\alpha^\text{v}.
\end{equation*}
By Lemma \ref{lem: hor and ver c-lift}, we have that $X^c=X^\text{h}-(\nabla X)^\upsilon$, hence $Q=\pi^h+X^c+\alpha^\text{v}$.

Let us now solve the last equation \eqref{eq: PW-3}. Plugging in the general solution of \eqref{eq: PW-1} and \eqref{eq: PW-2} yields
\begin{align*}
 0=\,&\pr^*\Big(\frac{\partial\alpha_i}{\partial x^j}+\frac{\partial \alpha_j}{\partial x^i}-2\alpha_k\Gamma^k_{ji}\Big)\\
 &+p_l\,\pr^*\Big(-\frac{\partial (\nabla X)^l_i}{\partial x^j}-\frac{\partial (\nabla X)^l_j}{\partial x^i}-(\nabla X)^k_i\Gamma^l_{kj}-(\nabla X)^k_j\Gamma^l_{ki}+2(\nabla X)^l_k\Gamma^k_{ji}\Big)\\
 &+p_l\,\pr^*(X^k((R_\nabla )^l_{ijk}+(R_\nabla )^l_{jik}))+p_lp_m\,\pr^*\pi^{mk}((R_\nabla )^l_{ijk}+(R_\nabla )^l_{jik}).
\end{align*}
Rewriting the first row in terms of the symmetric derivative and the second row in terms of the symmetric curvature operator (Definition \ref{def: sym-curv}), we get
\begin{equation*}
 0=\pr^*(\nabla^s\alpha)_{ij}-p_l\,\pr^*(R^s_\nabla X+2\sym \iota_X R_\nabla )^l_{ij}+p_lp_m\,\pr^*\pi^{mk}((R_\nabla )^l_{ijk}+(R_\nabla )^l_{jik}).
\end{equation*}
As the three terms contain different powers of the fibre coordinates, \eqref{eq: PW-3} is indeed equivalent to three independent conditions that can be written in a global form as
\begin{align*}
 \nabla^s\alpha&=0, & R^s_\nabla X+2\sym \iota_X R_\nabla &=0, & \sym(R_\nabla (Y,\pi(\,))Z+R_\nabla (Z,\pi(\,))Y)&=0.
\end{align*}
The first condition gives that $\alpha\in\kil^1\nb(M)$ and, by Proposition \ref{prop: affine vector field tf}, the second condition is equivalent to $X\in\aff(M,\nabla)$, which concludes the proof.
\end{proof}

\begin{remark}
Every $\pi\in\mathfrak{X}^2(M)$ such that $\pi^h$ is a Killing vector field for $g\nb$ is, in particular, a regular Poisson structure on $M$ as it is parallel for a torsion-free connection \cite[Sec. 11.4]{CraLPG}). In the following, we will use the notation:
 \begin{equation*}
 \mathfrak{X}^2(M,\nabla)\coloneqq \{\pi\in\mathfrak{X}^2(M)\,|\, \nabla\pi=0 \text{ and }R^\pi_\nabla=0\}.
 \end{equation*}
\end{remark}

\begin{example}
 For a symplectic structure $\omega$ on $M$, we can always find a torsion-free connection $\nabla$ on $M$ such that $\omega$ is parallel or, equivalently, $\omega^{-1}\in\mathfrak{X}^2(M)$ satisfies $\nabla\omega^{-1}=0$. By the non-degeneracy and the symmetries of the symplectic curvature tensor, we have that $\omega^{-1}\in\mathfrak{X}^2(M,\nabla)$ if and only if
 \begin{equation}\label{eq: symp-R}
 \omega(R_\nabla(X,Y)Z,W)=\omega(R_\nabla(W,Z)Y,X).
 \end{equation}
 By taking the contraction in $(Z,W)$ via $\omega$, \eqref{eq: symp-R} implies that $\nabla$ is Ricci flat, which gives an obstruction on a symplectic form to lift to a Killing vector field for $g\nb$.
\end{example}

The space of Killing vector fields for any (pseudo-)Riemannian metric forms a Lie subalgebra of $\mathfrak{X}(M)$. We show that, in the case of the Patterson-Walker metric, this Lie algebra admits the natural structure of a so-called \textit{$3$-graded Lie algebra}.

\begin{theorem}\label{thm: PW-Lie}
 Let $\nabla$ be a torsion-free connection on $M$. The decomposition in Theorem \ref{thm: Kill-PW} represents a natural grading of the Lie algebra of Killing vector fields for $g\nb$ concentrated in degrees $-1,0,1$. In particular, for $\alpha,\beta\in\kil^1(M,\nabla)$, $X,Y\in\aff(M,\nabla)$ and $\pi,\rho\in\mathfrak{X}^2(M,\nabla)$, we have
 \begin{align*}
 [\alpha^\emph{v},\beta^\emph{v}]&=0, & [X^c,\alpha^\emph{v}]&=(L_X\alpha)^\emph{v}, & [\alpha^\emph{v},\pi^h]&=\pi(\alpha)^c,\\
 [X^c,Y^c]&=[X,Y]^c, & [X^c,\pi^h]&=(L_X\pi)^h, & [\pi^h,\rho^h]&=0.
 \end{align*}
\end{theorem}

\begin{proof}
 By Lemma \ref{lem: Lie brackets hor and ver} and Proposition \ref{prop: c-lift Lie}, we already have, respectively, that $[\alpha^\text{v},\beta^\text{v}]=0$ and $ [X^c,Y^c]=[X,Y]^c$. For an arbitrary $Z\in\mathfrak{X}(M)$, we find that
 \begin{align*}
 [X^c,\alpha^\text{v}]Z^v&=X^c\pr^*\alpha(Z)-\alpha^\text{v}[X,Z]^v=\pr^*(X\alpha(Z)-\alpha([X,Z])=\pr^*(L_X\alpha)(Z)\\
 &=(L_X\alpha)^\text{v}Z^v,
 \end{align*}
 hence, by Proposition \ref{prop: vliftvfield}, we conclude that $[X^c,\alpha^\text{v}]=(L_X\alpha)^\text{v}$.
 
 Analogously, by Lemmas \ref{lem: hor and ver coord} and \ref{lem: pi pol} and Corollary \ref{cor: hor and ver pol}, we get
 \begin{align*}
 [\alpha^\text{v},\pi^h]Z^v&=\alpha^\text{v}(\nabla Z\circ \pi-\pi\circ(\nabla Z)^t)^v-\pi^h\pr^*\alpha(Z)\\
 &=(\nabla_{\pi(\alpha)}Z-\pi(\alpha(\nabla Z))+\pi(\dif\alpha(Z)))^v=(\nabla_{\pi(\alpha)}Z+\pi((\nabla\alpha)(Z)))^v.
 \end{align*}
 Using $\iota_Z\nabla^s\alpha=\nabla_Z\alpha+(\nabla \alpha)(Z)$ and $(\nabla_Z\pi)(\alpha)=\nabla_Z\pi(\alpha)-\pi(\nabla_Z\alpha)$, we obtain
 \begin{equation*}
 [\alpha^\text{v},\pi^h]Z^v=(\pi(\iota_Z\nabla^s\alpha)+(\nabla_Z\pi)(\alpha)-\nabla_Z\pi(\alpha)+\nabla_{\pi(\alpha)}Z)^v,
 \end{equation*}
 that is, by torsion-freeness and Proposition \ref{prop: vliftvfield}, we have
 \begin{equation*}
 [\alpha^\text{v},\pi^h]=(\pi(\nabla^s\alpha)+(\nabla\pi)(\alpha))^\upsilon+\pi(\alpha)^c
 \end{equation*}
 and the result follows because $\nabla^s\alpha=0$ and $\nabla\pi=0$.

By Lemmas \ref{lem: c-lift pol} and \ref{lem: pi pol}, we find that
\begin{align*}
 [X^c,\pi^h]Z^v&=X^c(\nabla Z\circ\pi-\pi\circ(\nabla Z)^t)^v-\pi^h[X,Z]^v\\
 &=(L_X(\nabla Z\circ\pi-\pi\circ(\nabla Z)^t)-\nabla[X,Z]\circ\pi+\pi\circ(\nabla[X,Z])^t)^v.
\end{align*}
The fact that $L_X\vartheta=L_X\circ\vartheta-\vartheta\circ L_x$ for $\vartheta\in\Gamma(\Sym^2TM)$ yields that $[X^c,\pi^h]Z^v$ is equal to $\kappa^v$ for $\kappa\in\Gamma(\Sym^2TM)$ given by its action on $\eta\in\Omega^1(M)$ as
\begin{align*}
 \kappa(\eta)=\,&[X,\nabla_{\pi(\eta)}Z]-\nabla_{\pi(L_X\eta)}Z-[X,\pi(\eta(\nabla Z))]\\
 &+\pi((L_X\eta)(\nabla Z))-\nabla_{\pi(\eta)}[X,Z]+\pi(\eta(\nabla[X,Z])).
\end{align*}
By Proposition \ref{prop: affine vf}
 and the fact that $X\in\aff(M,\nabla)$, we have, on the one hand,
\begin{equation*}
 [X,\nabla_{\pi(\eta)}Z]-\nabla_{\pi(L_X\eta)}Z-\nabla_{\pi(\eta)}[X,Z]=\nabla_{[X,\pi(\eta)]-\pi(L_X\eta)}Z=\nabla_{(L_X\pi)(\eta)}Z,
\end{equation*}
while, on the other hand, for $W\in\mathfrak{X}(M)$, we obtain
\begin{align*}
 (L_X\eta)(\nabla_W Z)+\eta(\nabla_W[X,Z])&=X\eta(\nabla_WZ)-\eta([X,\nabla_WZ]-\nabla_W[X,Z])\\
 &=X\eta(\nabla_WZ)-\eta(\nabla_{[X,W]}Z)=(L_X(\eta(\nabla Z)))(W).
\end{align*}
Therefore, $[X^c,\pi^h]Z^v=(\nabla Z\circ L_X\pi-L_X\pi\circ (\nabla Z)^t)^v$ because we have
\begin{equation*}
 \kappa(\eta)=\nabla_{(L_X\pi)(\eta)}Z-[X,\pi(\eta(\nabla Z))]+\pi(L_X(\eta(\nabla Z)))=\nabla_{(L_X\pi)(\eta)}Z-(L_X\pi)(\eta(\nabla Z))
\end{equation*}
and the result follows by Proposition \ref{prop: vliftvfield} and Lemma \ref{lem: pi pol}. 

Finally, by Lemma \ref{lem: pi pol}, we get that $[\pi^h, \rho^h]Z^v=\mathcal{X}^v$ for some $\mathcal{X}\in\Gamma(\Sym^3TM)$. By Theorem \ref{thm: Kill-PW}, the definition of the complete lift of a vector field, and Lemmas \ref{lem: hor and ver} and \ref{lem: pi pol}, we have that a Killing vector field $Q\in\mathfrak{X}(T^*M)$ satisfies
\begin{equation*}
 QZ^v=\mathcal{X}^v
\end{equation*}
for some $\mathcal{X}\in\Gamma(\Sym^3TM)$ if and only $Q=0$ (hence $\mathcal{X}=0$). Therefore, by the fact that Killing vector fields are closed with respect to the Lie bracket of vector fields, we conclude that $[\pi^h,\rho^h]=0$.
\end{proof}

\subsection{Simplified commutation relations of symmetric Cartan calculus
}\label{sec: simp-com}
In Section \ref{sec:commutation-relations}, we derived formulas \eqref{eq: com1} and \eqref{eq: com2} for the commutators $[L^s_X, L^s_Y]$ and $[\nabla^s, L^s_X]$. In Theorem \ref{thm: commutation-relations-sym-Cartan-PW}, we then identified the necessary and sufficient condition under which these commutators simplify. In this section, we provide a conceptual explanation for why such a condition arises in symmetric Cartan calculus, while it is not apparently needed in classical Cartan calculus.

\begin{definition}
 Let $\nabla$ be a connection on $M$. A vector field $Q\in\mathfrak{X}(T^*M)$ is called a \textbf{gradient Killing vector field} for $g\nb$ if there is $F\in\cCi(T^*M)$ such that $Q=g\nb^{-1}(\dif F)$ and, moreover, $Q$ is a Killing vector field for $g\nb$.
\end{definition}
 
\begin{lemma}\label{lem: comp Kill}
 Let $\nabla$ be a torsion-free connection on $M$. The complete lift of a vector field $X\in\mathfrak{X}(M)$ is a gradient Killing vector field for $g\nb$ if and only if $X\in\aff_0(M,\nabla)$.
\end{lemma}

\begin{proof}
By Theorem \ref{thm: commutation-relations-sym-Cartan}, $X^c$ is a Killing vector field for $g\nb$ if and only if $X\in\aff(M,\nabla)$. Using the fact that $X^c=X^\text{h}-(\nabla X)^\upsilon$ (see Lemma \ref{lem: hor and ver c-lift}) yields that $X^c$ is a gradient if and only if there is a function $ F\in\cCi(T^*M)$ such that
 \begin{align*}
 g\nb(X^\text{h})&=\rest{\dif F}{\Gamma(\mathcal{V})}, & g\nb(-(\nabla X)^\upsilon)&=\rest{\dif F}{\Gamma(\mathcal{H}\nb)}.
 \end{align*}
 Equivalently, by Lemmas \ref{lem: PW and End} and \ref{lem: hor and ver}, for any $\alpha\in\Omega^1(M)$ and $Y\in\mathfrak{X}(M)$,
 \begin{align*}
 \alpha^\text{v}(X^v- F)&=0, & Y^\text{h}(X^v+ F)&=0.
 \end{align*}
 It follows from the first equation that $ F=X^v+\pr^*f$ for some $f\in\cCi(M)$. Applying this to the second equation, we obtain
 \begin{equation*}
 2(\nabla_YX)^v+\pr^*(Yf)=0.
 \end{equation*}
 As the first term depends on the fibre coordinates and the second one does not, necessarily both of them vanish, that is, $X$ is parallel and $f$ is constant, so the result follows.
\end{proof}

\begin{theorem}\label{thm: commutation-relations-sym-Cartan-PW}
 Let $\nabla$ be a torsion-free connection on $M$ and $X,Y\in\mathfrak{X}(M)$. All of the six relations
 \begin{align*}
 [\iota_X,\iota_Y]&=0,& [\nabla^s,\nabla^s]&=0,&
 [\iota_X,\nabla^s]&=L^s_X,\\
 [L^s_X,\iota_Y]&=\iota_{[ X,Y]_s}, & [L^s_X,L^s_Y]&=L^s_{[X,Y]_s}, & [\nabla^s, L^s_X]&=0
\end{align*}
are satisfied if and only if $X^\text{c}$ is a gradient Killing vector field for $g\nb$.
\end{theorem}

\begin{proof}
 It is a direct consequence of Theorem \ref{thm: commutation-relations-sym-Cartan} and Lemma \ref{lem: comp Kill}.
\end{proof}






Having in mind the analogy between $g\nb$ and $\omega_\text{can}$, the condition of the `complete lift being gradient Killing' translates into the `complete lift being Hamiltonian'. However, by \eqref{eq:complete-is-hamiltonian} in Appendix \ref{app: lifts}, the latter condition is trivially satisfied. Consequently, we obtain a perfectly parallel analogy between symmetric Cartan calculus and classical Cartan calculus.

\begin{remark}
 Note that the analogy between gradient Killing and Hamiltonian is complete. Indeed, just as for a (pseudo-)Riemannian metric $g$, any gradient Killing vector field is automatically Killing ($L_Xg=0$), we have that for a symplectic form $\omega$, any Hamiltonian vector field is symplectic ($L_X\omega=0$).
\end{remark}

\appendix

\section{Non-geometric derivations}\label{app: non-geometric}
 We have taken the symmetric derivative $\nabla^s$ to be the starting point of symmetric Cartan calculus. Although we have explained why it is natural that $\nabla^s$ does not square to zero. A natural question is whether there are derivations of $\Upsilon^\bullet(M)$ of degree $1$, not necessarily geometric, that square to zero.

 The derivations of $\Upsilon^\bullet (M)$ were classified in \cite{HeydSD}. In particular, derivations of degree $1$ are parametrized by pairs $(A, \sigma)$ consisting of $A\in\Gamma(\en TM)$ and $\sigma\in\Upsilon^2(M,TM)$. Explicitly, any $D\in\der_1(\Upsilon^\bullet(M))$ is of the form
 \begin{equation*}
 D=[\iota^s_A,\nabla^s]+\iota^s_\sigma=\nabla^s_A+\iota^s_\sigma
 \end{equation*}
 for some auxiliary symmetric derivative $\nabla^s$. The geometric derivations are precisely those with $A=\id_{TM}$, see Remark \ref{rem: nablas_id} and Proposition \ref{prop: symCartan variation}.

 \begin{proposition}
 The only degree-$1$ derivation of $\Upsilon^\bullet(M)$ that squares to zero is the trivial one, $D=0$.
 \end{proposition}

 \begin{proof}
 Consider a derivation $D\in\der_1(\Upsilon^\bullet(M))$ described by a pair $(A,\sigma)$ as above. Assume that $D\circ D=0$. We then have
 \begin{equation*}
 0=D(Df^2)=D(2fDf)=2Df\odot Df,
 \end{equation*}
 hence $\rest{D}{\cCi(M)}=0$. In terms of the corresponding $(A,\sigma)$, it means $(AX)f=0$ for all $X\in\mathfrak{X}(M)$ and $f\in \cCi(M)$. Therefore, $A=0$ and $D=\iota^s_\sigma$. In addition,
 \begin{equation*}
 0=(D(D\alpha))(X,X,X)=3(\iota^s_\sigma\alpha)(\sigma(X,X),X)=3\alpha(\sigma(\sigma(X,X),X))
 \end{equation*}
 for all $\alpha\in\Upsilon^1(M)$, so $\sigma(\sigma(X,X),X)=0$. Finally, for every $C\in\Upsilon^2(M)$,
 \begin{align*}
 0&=(D(DC))(X,X,X,X)=(3!) (\iota^s_\sigma C)(\sigma(X,X),X,X)\\
 &=12C(\sigma(\sigma(X,X),X),X)+6C(\sigma(X,X),\sigma(X,X)).
 \end{align*}
 Since $\rest{D\circ D}{\Upsilon^1(M)}=0$, we get $C(\sigma(X,X),\sigma(X,X))=0$. If we choose $C$ to be a Riemannian metric, this yields $\sigma(X,X)=0$. By polarization, we get $\sigma=0$ and thus $D=0$. 
 \end{proof}

This proposition reinforces the claim that the symmetric derivative is indeed the natural analogue of the exterior derivative.


\section{Various lifts to the cotangent bundle}\label{app: lifts}

In this section, we recall several lifts from the base manifold to the total space of the cotangent bundle $\pr\colon T^*M\rightarrow M$ that are used in Section \ref{sec:Patterson-Walker}.

\subsection{Vertical lifts of $1$-forms and horizontal lifts of vector fields}

The \textbf{vertical subbundle} $\mathcal{V}=\ker \pr_*\subseteq T(T^*M)$ is a subbundle of rank $\dim M$ canonically isomorphic to $\pr^*T^*M$. The \textbf{vertical lift} of a $1$-form is defined as
\begin{align*}
 \Omega^1(M)&\to \Gamma(\mathcal{V})\subseteq \mathfrak{X}(T^*M)\\
 \alpha&\mapsto \alpha^\text{v}\coloneqq \pr^*\alpha,
 \end{align*}
where $\pr^*\alpha\in\Gamma(\pr^*T^*M)$ denotes the pullback section.

A connection $\nabla$ on $M$ induces, through the dual connection, the \textbf{horizontal subbundle} $\mathcal{H}\nb\leq T(T^*M)$, of rank $\dim M$ and satisfying $T(T^*M)=\mathcal{H}\nb\oplus \mathcal{V}$. It is canonically isomorphic to $\pr^*TM$, which gives the \textbf{horizontal lift}
\begin{align*}
 \mathfrak{X}(M)&\to \Gamma(\mathcal{H}\nb)\subseteq \mathfrak{X}(T^*M)\\
 X&\mapsto X^\text{h}\coloneqq \pr^* X.
\end{align*}

\begin{remark}\label{rem:coord-hor-vector}
 In natural local coordinates $(T^*U,\lbrace x^i\rbrace\cup\lbrace p_j\rbrace)$ on $T^*M$, we have
 \begin{align*}
 \rest{\alpha}{U}&=\alpha_j\dif x^j, & \rest{\alpha^\text{v}}{T^*U}&=(\pr^*\alpha_j)\partial_{p_j},\\
\rest{X}{U}&=X^i\partial_{x^i}, & \rest{X^\text{h}}{T^*U}&=(\pr^*X^i)(\partial_{x^i}+p_k(\pr^*\Gamma^k_{ji})\partial_{p_j}),
\end{align*}
where $\lbrace\Gamma^k_{ij}\rbrace$ is the set of Christoffel symbols of $\nabla$ corresponding to $(U,\lbrace x^j\rbrace)$.
\end{remark}

Using the coordinate expressions, we can prove a useful technical result.

\begin{lemma}\label{lem: hor and ver coord}
Let $\nabla$ be a connection on $M$. For $X\in\mathfrak{X}(M)$, $\alpha\in\Omega^1(M)$ and $f\in \cCi(M)$,
 \begin{align*}
 \alpha_\emph{can}(X^\emph{h})&=X^v, & X^\emph{h}(\pr^*f)&=\pr^*(Xf), & \alpha^\emph{v}(\pr^*f)&=0.
 \end{align*}
\end{lemma}

\begin{proof}
 It follows by a straightforward calculation in natural coordinates on $T^*M$.
\end{proof}

\subsection{Vertical lifts of symmetric multivector fields and fields of endomorphisms}
The \textbf{vertical lift of a vector field} is a map $(\,\,)^v\colon \mathfrak{X}(M)\rightarrow \cCi(T^*M)$ defined by the formula
\begin{equation*}
 X^v(\zeta)\coloneqq \zeta(X_{\pr(\zeta)})
\end{equation*}
for all $\zeta\in T^*M$. 
%
%
The importance of the vertical lift of a vector field can be seen through the following proposition.

\begin{proposition}[\cite{Patlift}]\label{prop: vliftvfield}
 A vector field on $T^*M$ is fully determined by its action on vertical lifts of all vector fields on $M$.
\end{proposition}

This concept is straightforwardly extended to the \textbf{vertical lift of a symmetric multivector field} $(\,\,)^v\colon\Gamma(\Sym^rTM)\rightarrow \cCi(T^*M)$ by
\begin{equation*}
 \mathcal{X}^v(\zeta)\coloneqq\frac{1}{r!}\mathcal{X}(\zeta\varlist\zeta).
\end{equation*}

Proposition \ref{prop: vliftvfield} allows us to give an elegant definition to the \textbf{vertical lift of a field of endomorphisms}, which is the map $(\,\,)^\upsilon\colon \Gamma(\en TM)\rightarrow\mathfrak{X}(T^*M)$ fully determined by the relation
\begin{equation*}
 A^\upsilon X^v\coloneqq (AX)^v.
\end{equation*}

\begin{remark}
 In natural local coordinates, for $\mathcal{X}\in\Gamma(\Sym^rTM)$ and $A\in\Gamma(\en TM)$,
\begin{align*}
 \rest{\mathcal{X}}{U}&=\frac{1}{r!}\mathcal{X}^{i_1\dots i_r}\partial_{x_{i_1}}\odot\dots\odot\partial_{x_{i_r}}, & \rest{\mathcal{X}^v}{T^*U}&=\frac{1}{r!}p_{i_1}\dots p_{i_r}(\pr^*\mathcal{X}^{i_1\dots i_r}),\\
 \rest{A}{U}&=A^i_j\dif{x^j}\otimes\partial_{x^i}, & \rest{A^\upsilon}{T^*U}&=p_i(\pr^*A^i_j)\partial_{p_j}.
\end{align*}
\end{remark}

Taking into account Proposition \ref{prop: vliftvfield}, we can write explicit formulas that fully determine the horizontal lift of a vector field and the vertical lift of a $1$-form.

\begin{lemma}\label{lem: hor and ver}
Let $\nabla$ be a connection on $M$. For $X,Y\in\mathfrak{X}(M)$ and $\alpha\in\Omega^1(M)$,
\begin{align*}
 X^\text{\emph{h}}Y^v&=(\nabla_XY)^v, & \alpha^\text{\emph{v}}Y^v&=\pr^*\alpha(Y).
\end{align*}
\end{lemma}

\begin{proof}
 In natural coordinates, using Lemma \ref{lem: hor and ver coord}, we have 
 \begin{align*}
 \rest{(X^\text{h}Y^v)}{T^*U}&=X^\text{h}(p_l(\pr^*Y^l))=p_l(\pr^*(X(Y^l))+(\pr^*X^i)(\pr^*Y^j)(\pr^*\Gamma^l_{ji}))\\
 &=\rest{(\nabla_XY)^v}{T^*U},
 \end{align*}
 where $\rest{X}{U}=X^i\partial_{x^i}$. On the other hand,
 \begin{equation*}
 \rest{(\alpha^\text{v}Y^v)}{T^*U}=(\pr^*\alpha_j)\partial_{p_j}(p_l(\pr^*Y^l))=(\pr^*\alpha_j)(\pr^*Y^j)=\rest{\pr^*\alpha(Y)}{T^*U},
 \end{equation*}
 where $\rest{\alpha}{U}=\alpha_j\dif x^j$. Thus, also globally, $X^\text{h}Y^v=(\nabla_XY)^v$ and $\alpha^\text{v}Y^v=\pr^*\alpha(Y)$.
 \end{proof}


 By the Leibniz rule, we extend the result of Lemma \ref{lem: hor and ver}.

 \begin{corollary}\label{cor: hor and ver pol}
 Let $\nabla$ be a connection on $M$. For $X\in\mathfrak{X}(M)$, $\alpha\in\Omega^1(M)$, and $\mathcal{X}\in\Gamma(\Sym^r TM)$, we have
 \begin{align*}
 X^\emph{h}\mathcal{X}^v&=(\nabla_X\mathcal{X})^v, & \alpha^\emph{v}\mathcal{X}^v&=(\iota_\alpha\mathcal{X})^v.
 \end{align*}
 \end{corollary}
 
Using Proposition \ref{prop: vliftvfield}, we easily compute the Lie brackets between the lifts.

\begin{lemma}\label{lem: Lie brackets hor and ver}
 Let $\nabla$ be a connection on $M$. For $X,Y\in\mathfrak{X}(M)$ and $\alpha,\beta\in\Omega^1(M)$,
 \begin{align*}[X^\text{\emph{h}},Y^\text{\emph{h}}]&=[X,Y]^\text{\emph{h}}+R_\nabla (X,Y)^\upsilon, & [X^\text{\emph{h}},\alpha^\text{\emph{v}}]&=(\nabla_X\alpha)^\text{\emph{v}}, & [\alpha^\text{\emph{v}},\beta^\text{\emph{v}}]&=0.
 \end{align*} 
\end{lemma}

\begin{proof}
It follows from Lemmas \ref{lem: hor and ver coord} and \ref{lem: hor and ver}, for $Z\in\mathfrak{X}(M)$, that
\begin{align*}
 [X^\text{h}, Y^\text{h}]Z^v&=(\nabla_X\nabla_YZ-\nabla_Y\nabla_XZ)^v=(\nabla_{[X,Y]}Z+R_\nabla (X,Y)Z)^\upsilon\\
 &=([X,Y]^\text{h}+R_\nabla (X,Y)^\upsilon)Z^v\\
[X^\text{h},\alpha^\text{v}]Z^v&=X^\text{h}\pr^*\alpha(Z)-\pr^*\alpha(\nabla_XZ)=\pr^*(\nabla_X\alpha)(Z)=(\nabla_X\alpha)^\text{v}Z^v\\
[\alpha^\text{v}, \beta^\text{v}]Z^v&=\alpha^\text{v}\pr^*\beta(Z)-\beta^\text{v}\pr^*\alpha(Z)=0.
\end{align*}
\end{proof}

\subsection{Complete lifts of vector fields}In addition to the vertical and horizontal lifts, one can consider also the \textbf{complete lift}. It is the map $(\,\,)^c\colon \mathfrak{X}(M)\rightarrow\mathfrak{X}(T^*M)$ determined by the relation
\begin{equation*}
 X^cY^v\coloneqq [X,Y]^v.
\end{equation*}

\begin{lemma}\label{lem: c-lift pol}
 For $X\in\mathfrak{X}(M)$ and $\mathcal{X}\in\Gamma(\Sym^rTM)$, we have
 \begin{equation*}
 X^c\mathcal{X}^v=(L_X\mathcal{X})^v.
 \end{equation*}
\end{lemma}

\begin{proof}
 It follows directly from the definition of $X^c$ and the Leibniz rule.
\end{proof}

Alternatively, the complete lift of $X\in\mathfrak{X}(M)$ can be defined (up to sign) as the Hamiltonian vector field of the vertical lift of $X$:
\begin{equation}\label{eq:complete-is-hamiltonian}
 X^c=-\omega^{-1}_\text{can}(\dif X^v)=-\ham X^v.
\end{equation}

Let us recall two important properties of the complete lift.

\begin{proposition}[\cite{Patlift}]\label{prop-determined-by-complete}
 A tensor field that is a section of $\otimes^r T^*M\otimes TM$ or $\otimes^rT^*M$ is fully determined by its values on complete lifts of vector fields on $M$.
\end{proposition}


\begin{proposition}[\cite{Patlift}]\label{prop: c-lift Lie}
 For every $X,Y\in\mathfrak{X}(M)$, we have $[X^c,Y^c]=[X,Y]^c$.
\end{proposition}

\begin{remark}
In natural local coordinates, for $X\in\mathfrak{X}(M)$,
 \begin{align*}
 \rest{X}{U}&=X^i\partial_{x^i}, & \rest{X^c}{T^*U}&=(\pr^*X^i)\partial_{x^i}-p_i\Big(\pr^*\frac{\partial X^i}{\partial x^j}\Big)\partial_{p_j}.
 \end{align*}
\end{remark}

If a connection is given, we can express the complete lift in terms of the horizontal and vertical lifts.

\begin{lemma}\label{lem: hor and ver c-lift}
 Let $\nabla$ be a connection on $M$. For $X\in\mathfrak{X}(M)$, we have
 \begin{equation*}
 X^c=X^\emph{h}-(\nabla X)^\upsilon.
 \end{equation*}
\end{lemma}

\begin{proof}
 It follows by a straightforward calculation in natural coordinates on $T^*M$.
\end{proof}

\subsection{Horizontal lifts of bivector fields}
We finish by giving a coordinate-free description of the lift of a bivector field $\pi\in\mathfrak{X}^2(M)$ introduced in \cite{conformal-PW}.

\begin{definition}
 Let $\nabla$ be a connection on $M$. The \textbf{horizontal lift of a bivector field} is the map
 \begin{align*}
 (\,)^h\colon \,\mathfrak{X}^2(M)&\rightarrow\mathfrak{X}(T^*M)\\
 \pi&\mapsto\pi^h
 \end{align*}
 uniquely determined, for $f\in\cCi(M)$ and $X,Y\in\mathfrak{X}(M)$, by the properties
 \begin{align*}
 &(f\pi)^h=(\pr^*f)\pi^h, & &(X\wedge Y)^h=Y^vX^\text{h}-X^vY^\text{h}.
 \end{align*}
\end{definition}

It is clear right from the definition that the horizontal lift of a bivector field is indeed horizontal. We give an equivalent characterization of $\pi^h$ using the Patterson-Walker metric (Definition \ref{def:PW-metric}).

\begin{proposition}\label{prop: pi-lift}
 Let $\nabla$ be a connection on $M$. The horizontal lift of a bivector field $\pi\in\mathfrak{X}^2(M)$ is the unique horizontal vector field $\pi^h$ on $T^*M$ such that, for every $\alpha\in\Omega^1(M)$, we have
 \begin{equation}\label{eq: pi-lift-global}
 g\nb(\pi^h,\alpha^\emph{v})=\pi(\alpha)^v.
 \end{equation}
\end{proposition}

\begin{proof}
 For every $f\in\cCi(M)$, we have
 \begin{equation*}
 (f\pi)(\alpha)^v=(f\pi(\alpha))^v=(\pr^*f)\pi(\alpha)^v,
 \end{equation*}
 so it is enough to check that \eqref{eq: pi-lift-global} is satisfied on decomposable elements of $\mathfrak{X}^2(M)$. Indeed, for $X,Y\in\mathfrak{X}(M)$, we get
 \begin{align*}
 g\nb((X\wedge Y)^h,\alpha^\text{v})&=Y^vg\nb(X^\text{h},\alpha^\text{v})-X^vg\nb(Y^\text{h},\alpha^\text{v})=Y^v\pr^*\alpha(X)-X^v\pr^*\alpha(Y)\\
 &=(\alpha(X)Y-\alpha(Y)X)^v=(X\wedge Y)(\alpha)^v.
 \end{align*}
 The condition \eqref{eq: pi-lift-global} uniquely determines $\pi^h$ since the Patterson-Walker metric is non-degenerate and the horizontal subbundle is isotropic for $g\nb$. 
\end{proof}

\begin{remark}\label{rem:coord-hor-bivector}
 By a straightforward calculation, one finds that, in natural coordinates, we have
 \begin{align*}
 &\rest{\pi}{U}=\frac{1}{2}\pi^{ij}\partial_{x^i}\wedge\partial_{x^j}, & &\rest{\pi^h}{T^*U}=p_i(\pr^*\pi^{ij})\partial_{x^j}^\text{h}=p_i(\pr^*\pi^{ij})(\partial_{x^j}+p_l(\pr^*\Gamma^l_{kj})\partial_{p_k}).
 \end{align*}
This recovers the original definition in \cite{conformal-PW}. From the coordinate expression, we can also see that the vector field $\pi^h$ is the \textit{geodesic spray of the contravariant connection} $\nabla^\pi$ given, for $\alpha\in\Omega^1(M)$, by $\nabla^\pi_\alpha\coloneqq \nabla_{\pi(\alpha)}$, as defined in \cite{CraLPG}.
\end{remark}

We establish several properties that we need for the proof of Theorem \ref{thm: PW-Lie}.

\begin{lemma}\label{lem: pi pol}
 Let $\nabla$ be a connection on $M$. For $\pi\in\mathfrak{X}^2(M)$, $f\in\cCi(M)$, $X\in\mathfrak{X}(M)$, and $\vartheta\in\Gamma(\Sym^2TM)$, we have
 \begin{align*}
 \pi^h\pr^*f&=-\pi(\dif f)^v, & \pi^hX^v&=(\nabla X\circ \pi-\pi\circ(\nabla X)^t)^v, & \pi^h\vartheta^v=(\nabla_{\pi(\,)}\vartheta+\cyc)^v.
 \end{align*}
\end{lemma}

\begin{proof}
 By Lemmas \ref{lem: hor and ver coord} and \ref{lem: hor and ver}, we obtain the following in natural coordinates:
 \begin{align*}
 \rest{(\pi^h\pr^*f)}{T^*U}&=p_i(\pr^*\pi^{ij})\partial_{x_j}^\text{h}\pr^*f=p_i(\pr^*\pi^{ij})\Big(\pr^*\frac{\partial f}{\partial x^j}\Big)=\rest{(-\pi(\dif f))^v}{T^*U},\\
 \rest{(\pi^hX^v)}{T^*U}&=p_i(\pr^*\pi^{ij})\partial_{x_j}^\text{h}X^v=p_i(\pr^*\pi^{ij})(\nabla_{\partial_{x^j}}X)^v=p_ip_k\pr^*(\pi^{ij}(\nabla X)^k_j)\\
 &=\frac{1}{2}p_ip_k\pr^*((\nabla X)^i_j\pi^{kj}-\pi^{ji}(\nabla X)^k_j)
 \end{align*}
 where $\rest{\pi}{U}=\frac{1}{2}\pi^{ij}\partial_{x^i}\wedge\partial_{x^j}$ and $\rest{X}{U}=X^i\partial_{x^i}$. Therefore, $\pi^h\pr^*f=-\pi(\dif f)^v$, and $\pi^hX^v=\kappa^v$ for $\kappa\in\Gamma(\Sym^2TM)$ given by
 \begin{equation*}
 \kappa(\alpha,\beta)\coloneqq \beta(\nabla_{\pi(\alpha)}X)+\alpha(\nabla_{\pi(\beta)}\alpha)=(\nabla X\circ\pi-\pi\circ(\nabla X)^t)(\alpha,\beta).
 \end{equation*}

 Using the Leibniz rule and Lemmas \ref{lem: hor and ver coord} and \ref{lem: hor and ver}, we find that
 \begin{align*}
 \rest{(\pi^h\vartheta^v)}{T^*U}&=p_k(\pr^*\pi^{kl})\partial_{x^l}^\text{h}\Big(\frac{1}{2}p_ip_j(\pr^*\vartheta^{ij})\Big)\\
 &=\frac{1}{2}p_ip_jp_k\pr^*\Big(\pi^{il}\Big(\frac{\partial \vartheta^{jk}}{\partial x^l}+\vartheta^{qk}(\nabla_{\partial_{x^l}}\partial_{x^q})^j+\vartheta^{jq}(\nabla_{\partial_{x^l}}\partial_{x^q})^k\Big)\Big),
 \end{align*}
 hence $\pi^h\vartheta^v=\mathcal{X}^v$ for $\mathcal{X}\in\Gamma(\Sym^3TM)$ given by 
 \begin{equation*}
 \mathcal{X}(\alpha,\beta,\eta)\coloneqq(\nabla_{\pi(\alpha)}\vartheta)(\beta,\eta)+\cyc(\alpha,\beta,\eta).
 \end{equation*}
\end{proof}

\newgeometry{left=1.85cm, right=1.85cm, bottom=2cm}
\section{Comparison with classical Cartan calculus}\label{sec:table}
We show in this table a comparison between classical and symmetric Cartan calculus.
\begin{center}
\SetTblrInner{rowsep=7pt}
\begin{tblr}{width=\linewidth, colspec={|X[c]||X[c]|}, cell{2}{1}={c=2}{c}, cell{7}{1}={c=2}{c},cell{15}{1}={c=2}{c}, cell{18}{1}={c=2}{c}}
\hline
\textbf{Classical Cartan calculus} & \textbf{Symmetric Cartan calculus}\\
\hline\hline
algebraic features\\
\hline
graded-commutative algebra $(\Omega^\bullet(M),\wedge)$ & commutative graded algebra $(\Upsilon^\bullet(M),\odot)$\\
 graded derivations $\gder(\Omega^\bullet (M))$ & derivations $\der(\Upsilon^\bullet (M))$\\
graded commutator $[\,\,,\,\,]_\text{g}$ & commutator $[\,\,,\,\,]$\\
$[\iota_X,\iota_Y]_\text{g}=0$ & $[\iota_X,\iota_Y]=0$\\
\hline
\hline
differentials&\\
\hline
exterior derivative $\dif$ & symmetric derivative $\nabla^s$\\
$\dif\psi=(r+1)\,\mathrm{skew}\,(\nabla\psi)$ & $\nabla^s\varphi= (r+1)\,\mathrm{sym}\,(\nabla\varphi)$\\
canonical & depending on the choice of $\nabla$\\
$(\dif f)(X)=Xf$ & $(\nabla^s f)(X)=Xf$\\
$\dif\in\gder_1(\Omega^\bullet(M))$ & $\nabla^s\in\der_1(\Upsilon^\bullet(M))$\\
$[\dif,\dif]_\text{g}=2\,\dif\circ\dif=0$\vspace{5pt}\break (non-trivial condition)& $[\nabla^s,\nabla^s]=0$\vspace{5pt}\break (trivially satisfied)\\
$\ker\dif=\Omega^\bullet_\text{closed}(M)$ (closed forms)& $\ker\nabla^s=\kil^\bullet(M,\nabla)$ (Killing tensors)\\
\hline
\hline
Lie derivatives&\\
\hline
Lie derivative $L_X\coloneqq [\iota_X,\dif]_\text{g}$ & symmetric Lie derivative $L^s_X\coloneqq [\iota_X,\nabla^s]$\\ 
$\displaystyle (L_X\psi)_m=\rest{\frac{\dif}{\dif t}}{t=0}(\Psi^X_{t})_m^*\psi_{\Psi^X_t(m)}$ & $\displaystyle (L^s_X\varphi)_m=\rest{\frac{\dif}{\dif t}}{t=0}P^\gamma_{2t,0}(\Psi^X_{-t})_{\Psi^X_{2t}(m)}^*\varphi_{\Psi^X_t(m)}$\\
\hline
\hline
brackets&\\
\hline
Lie bracket $[X,Y]\coloneqq X\circ Y-Y\circ X$ & symmetric bracket $[X,Y]_s\coloneqq \nabla_XY+\nabla_YX$\\
$\iota_{[X,Y]}=[L_X,\iota_Y]_\text{g}$,\,\,\, \,\,$[X,Y]=L_XY$ & $\iota_{[X,Y]_s}=[L^s_X, \iota_Y]$,\,\,\,\,\,$[X,Y]_s=L^s_XY$ \\
foliations & geodesically invariant distributions\\
\hline
\end{tblr}
\end{center}

\begin{center}
\SetTblrInner{rowsep=7.5pt}
\begin{tblr}{width=\linewidth, colspec={|X[c]||X[c]|}, cell{2}{1}={c=2}{c}, cell{5}{1}={c=2}{c}, cell{14}{1}={c=2}{c}, cell{16}{1}={c=2}{c}}
\hline
\textbf{Classical Cartan calculus} & \textbf{Symmetric Cartan calculus}\\
\hline\hline
(graded-)commutation relations\\
\hline
$[L_X,L_Y]_\text{g}=L_{[X,Y]}$ & $[L^s_X, L^s_Y]=L^s_{[X,Y]}+\iota^s_{2(\nabla_{\nabla X}Y-\nabla_{\nabla Y}X-R_\nabla (X,Y))}$\\
$[\dif,L_X]_\text{g}=0$ & $[\nabla^s,L^s_X]=2\nabla^s_{\nabla X}+\iota^s_{2\sym \iota_XR_\nabla +R^s_\nabla X}$\\
\hline\hline
isomorphisms\\
\hline
diffeomorphisms& affine diffeomorphisms\\
\hline
\hline
induced geometry on the cotangent bundle&\\
\hline
symplectic geometry & geometry of split signature metrics\\
canonical symplectic form $\omega_\text{can}$ & Patterson-Walker metric $g\nb$\\
$\rest{\omega_\text{can}}{T^*U}=\dif p_j\wedge\dif x^j$ & $\rest{g\nb}{T^*U}=\dif p_j\odot\dif x^j-p_k(\pr^*\Gamma^k_{ij})\dif x^i\odot\dif x^j$\\
$\omega_\text{can}=\dif\alpha_{\text{can}}$ &$g\nb=\hat{\nabla}^s\alpha_\text{can}$\\
$\omega_\text{can}(X^c,Y^c)=-[X,Y]^v$ &$g\nb(X^c,Y^c)=-\pg{X,Y}^v_s$\\
$\omega_\text{can}(\phi\nb a,\phi\nb b)=\pr^*\la a,b\ra_-$& $g\nb(\phi\nb a,\phi\nb b)=\pr^*\la a, b\ra_+$\\
\hline
\hline
simplified (graded-)commutation relations\\
\hline
$[L_X,L_Y]_\text{g}=L_{[X,Y]}$,\,\,\, \,\,$[\dif,L_X]_\text{g}=0$\break $\Leftrightarrow$\break $X^c$ is Hamiltonian for $\omega_\text{can}$\break (trivially satisfied) & $[L^s_X, L^s_Y]=L^s_{[X,Y]_s}$,\,\,\, \,\,$[\nabla^s,L^s_X]=0$\break $\Leftrightarrow$\break $X^c$ is gradient Killing for $g\nb$\break (non-trivial condition)\\
\hline
\end{tblr}
\end{center}


\clearpage
\bibliographystyle{alpha}\bibliography{refs}

\newcommand{\etalchar}[1]{$^{#1}$}
\begin{thebibliography}{HS{\v S}{\etalchar{+}}19}

\bibitem[AX]{AleDBCA}
A.~Alekseev and P.~Xu.
\newblock Derived brackets and {C}ourant algebroids.
\newblock \url{https://web.archive.org/web/20160705150153/http://www.math.psu.edu/ping/anton-final.pdf}.
\newblock Unpublished notes.

\bibitem[Ben16]{BenSRM}
Sergio Benenti.
\newblock Separability in {R}iemannian manifolds.
\newblock {\em SIGMA Symmetry Integrability Geom. Methods Appl.}, 12:Paper No. 013, 21, 2016.

\bibitem[BL05]{Lewbook}
Francesco Bullo and Andrew~D. Lewis.
\newblock {\em Geometric control of mechanical systems}, volume~49 of {\em Texts in Applied Mathematics}.
\newblock Springer-Verlag, New York, 2005.
\newblock Modeling, analysis, and design for simple mechanical control systems.

\bibitem[BLnL12]{LewGISP}
Mar\'ia Barbero-Li\~n\'an and Andrew~D. Lewis.
\newblock Geometric interpretations of the symmetric product in affine differential geometry and applications.
\newblock {\em Int. J. Geom. Methods Mod. Phys.}, 9(8):1250073, 33, 2012.

\bibitem[Car68a]{CarA}
B.~Carter.
\newblock Global structure of the {K}err family of gravitational fields.
\newblock {\em Physical Review}, \textbf{174}(5):1559--1571, 1968.

\bibitem[Car68b]{CarB}
Brandon Carter.
\newblock Hamilton-{J}acobi and {S}chr\"odinger separable solutions of {E}instein's equations.
\newblock {\em Comm. Math. Phys.}, 10:280--310, 1968.

\bibitem[CFM21]{CraLPG}
Marius Crainic, Rui~Loja Fernandes, and Ioan Mărcuț.
\newblock {\em Lectures on {P}oisson geometry}, volume 217 of {\em Graduate Studies in Mathematics}.
\newblock American Mathematical Society, Providence, RI, 2021.

\bibitem[Cou94]{CouTLA}
Ted Courant.
\newblock Tangent {L}ie algebroids.
\newblock {\em J. Phys. A}, 27(13):4527--4536, 1994.

\bibitem[Cro81]{CroGSST}
P.~E. Crouch.
\newblock Geometric structures in systems theory.
\newblock {\em Proc. IEE-D}, 128(5):242--252, 1981.

\bibitem[DM18]{dunajski-mettler}
Maciej Dunajski and Thomas Mettler.
\newblock Gauge theory on projective surfaces and anti-self-dual {E}instein metrics in dimension four.
\newblock {\em J. Geom. Anal.}, 28(3):2780--2811, 2018.

\bibitem[Eas05]{EasHSL}
Michael Eastwood.
\newblock Higher symmetries of the {L}aplacian.
\newblock {\em Ann. of Math. (2)}, 161(3):1645--1665, 2005.

\bibitem[Gua04]{GuaGCG}
Marco Gualtieri.
\newblock Generalized complex geometry.
\newblock {\em Oxford University doctoral thesis}, [arXiv:0401221], 2004.

\bibitem[Gua10]{gualtieri-branes}
Marco Gualtieri.
\newblock Branes on {P}oisson varieties.
\newblock In {\em The many facets of geometry}, pages 368--394. Oxford Univ. Press, Oxford, 2010.

\bibitem[HBP06]{HeydSD}
A.~Heydari, N.~Boroojerdian, and E.~Peyghan.
\newblock A description of derivations of the algebra of symmetric tensors.
\newblock {\em Arch. Math. (Brno)}, 42(2):175--184, 2006.

\bibitem[Hit03]{HitGCYM}
Nigel Hitchin.
\newblock Generalized {C}alabi-{Y}au manifolds.
\newblock {\em Q. J. Math.}, 54(3):281--308, 2003.

\bibitem[HS{\v S}{\etalchar{+}}19]{conformal-PW}
Matthias Hammerl, Katja Sagersching, Josef {\v S}ilhan, Arman Taghavi-Chabert, and Vojtěch {\v Z}ádník.
\newblock Conformal {P}atterson-{W}alker metrics.
\newblock {\em Asian J. Math.}, 23(5):703--734, 2019.

\bibitem[KN96]{KoNoFDG}
Shoshichi Kobayashi and Katsumi Nomizu.
\newblock {\em Foundations of differential geometry. {V}ol. {I}}.
\newblock Wiley Classics Library. John Wiley \& Sons, Inc., New York, 1996.
\newblock Reprint of the 1963 original, A Wiley-Interscience Publication.

\bibitem[Kob95]{KobTGDG}
Shoshichi Kobayashi.
\newblock {\em Transformation groups in differential geometry}.
\newblock Classics in Mathematics. Springer-Verlag, Berlin, 1995.
\newblock Reprint of the 1972 edition.

\bibitem[Lew98]{LewACD}
Andrew~D. Lewis.
\newblock Affine connections and distributions with applications to nonholonomic mechanics.
\newblock {\em Rep. Math. Phys.}, 42(1-2):135--164, 1998.
\newblock Pacific Institute of Mathematical Sciences Workshop on Nonholonomic Constraints in Dynamics (Calgary, AB, 1997).

\bibitem[MR25]{SymPoisson}
Filip Moučka and Roberto Rubio.
\newblock Symmetric {P}oisson geometry, totally geodesic foliations and {J}acobi-{J}ordan algebras.
\newblock arXiv:2508.15890, 2025.

\bibitem[PW52]{PatRE}
E.~M. Patterson and A.~G. Walker.
\newblock Riemann extensions.
\newblock {\em Quart. J. Math. Oxford Ser. (2)}, 3:19--28, 1952.

\bibitem[WP70]{WalOQ}
Martin Walker and Roger Penrose.
\newblock On quadratic first integrals of the geodesic equations for type {$\{22\}$} spacetimes.
\newblock {\em Comm. Math. Phys.}, 18:265--274, 1970.

\bibitem[YP67]{Patlift}
K.~Yano and E.~M. Patterson.
\newblock Vertical and complete lifts from a manifold to its cotangent bundle.
\newblock {\em J. Math. Soc. Japan}, 19:91--113, 1967.

\end{thebibliography}

\end{document}